\documentclass[a4paper,12pt]{article}

\usepackage[english]{babel}
\usepackage[utf8]{inputenc}  
\usepackage[T1]{fontenc}
\usepackage{amsfonts,amsmath,amssymb,amsthm}
\usepackage[all]{xy}    
\usepackage{mathtools}  
\usepackage{multirow}    
\usepackage{csquotes}
 \usepackage{relsize}
 \usepackage{mdframed}
\usepackage{tikz}   
\usetikzlibrary{arrows}
\usetikzlibrary{shapes}
\usepackage{hyperref}\hypersetup{
     colorlinks=true, 
     bookmarksopen=true,            
     pdftitle={Integral methods for A2(2)}, 
     pdfauthor={Josha Box and Samuel Le Fourn},     
}
\usepackage{setspace}

\allowdisplaybreaks

\theoremstyle{plain}
\newtheorem{theorem}{Theorem}
\newtheorem{proposition}[theorem]{Proposition}
\newtheorem{corollary}[theorem]{Corollary}

\newtheorem{lemma}[theorem]{Lemma}

\theoremstyle{definition}
\newtheorem{definition}[theorem]{Definition}

\theoremstyle{remark}
 \newtheorem{remark}[theorem]{Remark}

\newcommand{\eps}{\epsilon}
\newcommand{\ep}{\epsilon}
\newcommand{\al}{\alpha}

\newcommand{\gm}{{\mathfrak{m}}}

\newcommand{\gP}{{\mathfrak{P}}}

\newcommand{\Bcal}{{\mathcal B}}
\newcommand{\Ccal}{{\mathcal C}}
\newcommand{\Dcal}{{\mathcal D}}

\newcommand{\Fcal}{{\mathcal F}}

\newcommand{\Hcal}{{\mathcal H}}

\newcommand{\Ncal}{{\mathcal N}}
\newcommand{\Ocal}{{\mathcal O}}

\newcommand{\Z}{{\mathbb{Z}}}
\newcommand{\Q}{{\mathbb{Q}}}

\newcommand{\C}{{\mathbb{C}}}
\renewcommand{\P}{\mathbb{P}}
\newcommand{\F}{\mathbb{F}}

\newcommand{\Id}{\operatorname{Id}}
\renewcommand{\Im}{\operatorname{Im}}

\newcommand{\Jac}{\operatorname{Jac}}

\newcommand{\Sp}{\operatorname{Sp}}

\renewcommand{\mod}{\, \operatorname{mod} \,}

\newcommand{\mt}{\mapsto}

\newcommand{\fonction}[5]{\begin{array}{c|ccl}           
#1: & #2 & \longrightarrow & #3 \\
    & #4 & \longmapsto & #5 \end{array}}

\newcommand{\Kb}{\overline{K}}

\newcommand{\um}{{\underline{m}}}

\newmdenv[
  topline=false,
  bottomline=false,
  linecolor=blue,
  skipabove=\topsep,
  skipbelow=\topsep
]{sideruleb}

\newmdenv[
  topline=false,
  bottomline=false,
  linecolor=violet,
  skipabove=\topsep,
  skipbelow=\topsep
]{siderulev}

\reversemarginpar{}

\title{Bounding integral points on the Siegel modular variety $A_2(2)$}
\author{Josha Box and Samuel Le Fourn}
\date\today

\begin{document}
	\maketitle
	\begin{abstract}
	   We determine two explicit upper bounds for the stable Faltings height of principally polarised abelian surfaces over number fields corresponding to $S$-integral points on the Siegel modular variety $A_2(2)$. One upper bound, using Runge's method, is uniform in $S$ as long as $|S|<3$; the other, using Baker's method, is not uniform but allows $|S|<10$.
	   Our application of a higher-dimensional Baker's method is completely explicit and improves upon what would be obtained from the general case due to Levin.
	\end{abstract}

\section{Introduction}

Diophantine methods, commonly used to determine the rational or integral points on algebraic varieties (i.e. to solve diophantine equations), tend to become particularly difficult when the underlying algebraic variety is not a curve; even more so when one desires effective results. 

For example, the Siegel varieties $A_g(n)$, which are the moduli spaces of $g$-dimensional principally polarised abelian varieties with full $n$-torsion, are widely studied but still quite mysterious for $g>1$. Important results include non-effective uniform boundedness of the full torsion under Vojta's conjecture \cite{AbramovichVarillyAlvarado17}, and effective height bounds for jacobians of hyperelliptic curves with good reduction \cite{vonKanel14}. We note that the latter result is of a different nature since it does not directly use the geometry of $A_g(n)$. 

The study of integral points on varieties was given a boost by Aaron Levin, who succeeded in extending Runge's method \cite{Levin08} (updated in \cite{Levin18}) and Baker's method \cite{Levin14} -- both classical methods for determining integral points on curves -- to varieties of any dimension. Recently, the second named author expanded on both methods and applied these to $A_2(2)$ \cite{LeFourn3}, \cite{LeFourn5}. In this article we push these methods to their limit to find stronger explicit effective results for integral points on $A_2(2)$. This is the first known explicit application of a higher-dimensional Baker's method. We hope this article paves the way for the explicit application of higher-dimensional Baker and Runge to other (modular) varieties. 

Our objects of interest are abelian surfaces over number fields with full 2-torsion (so $g=n=2$). Recall \cite{OortUeno} that over any field $k$, a principally polarised abelian surface $A_{/k}$ is isomorphic over a finite extension to either the jacobian of a (smooth) hyperelliptic curve of genus 2, or the product of elliptic curves (both endowed with the natural associated polarisations). 

When we consider an abelian surface $A$ over a number field $K$, its semistable reduction at every finite place $v$ of $K$ can be an abelian surface (potentially good reduction) or not (potentially multiplicative reduction). For our purposes, principally polarised abelian surface will be considered to reduce \enquote{nicely} at $v$ if the semistable reduction is not only an abelian surface, but also is isomorphic to the jacobian of a hyperelliptic curve of genus 2 over some finite extension. If we start with $A = \operatorname{Jac}(C)$ for some genus 2 hyperelliptic curve $C$, this is equivalent to saying that $C$ itself has potentially good reduction at $v$.

As we will see below, abelian surfaces which \enquote{reduce nicely} in this sense outside of a set of places $S$ correspond to $S$-integral points on $(A_2(2) \backslash D)$ for a certain divisor $D$, in a way which will be made precise later. This is the fundamental reason for which we adopt this interpretation. Furthermore, $A_2(2)$ has explicit (and workable) equations as a subvariety of $\P^9$, which makes it a good example to practise precise computations and methods.

Our two main results provide different insights on the effective finiteness of the integral points when the set $S$ of bad places is sufficiently small.

\begin{theorem}
\label{rungethm}
Let $(A,\lambda)$ be a principally polarised abelian surface defined over a number field $K$, with full 2-torsion defined over $K$. Consider the set $S$ of places $v$ of $K$ which are infinite or such that the semistable reduction of $A$ modulo $v$ is not isomorphic to the jacobian of a genus 2 curve. 

If $|S|<3$, we have a bound
\[
h_\Fcal(A) \leq 985,
\]
where $h_\Fcal$ is the stable Faltings height of $A$.
\end{theorem}
%

In fact, we compute a smaller bound for the Weil height with respect to a model for $A_2(2)$ (see Theorem \ref{rungethm2}). Choosing $K=\Q$ and $S=\{p,\infty\}$, we can search for $\Q$-rational points on $A_2(2)$ of height up to this bound, and obtain the following consequence (proven in Section \ref{subsecthmRungeandcorollary}).
\begin{corollary}
\label{rungecor}
 There is no genus 2 hyperelliptic curve $C$ over $\Q$ such that all Weierstrass points of $C$ are rational and $C$ has potentially good reduction at all but one of the primes.  \end{corollary}


The proof of Theorem \ref{rungethm} uses Runge's method in higher dimensions, as introduced by Levin \cite{Levin18} and first applied to $A_2(2)$ in \cite{LeFourn3}. Note that a similar bound is obtained in \cite[Theorem 8.2]{LeFourn3}, but there the condition is $|S|<2$, whereas we allow $|S|<3$. This means in practice that our set $S$ of ``bad'' places is allowed to contain a finite place, thus providing a significant improvement to \cite[Theorem 8.2]{LeFourn3}, as witnessed by Corollary \ref{rungecor}. 

Apart from Runge's method, we also execute a higher-dimensional version of Baker's method, which needs weaker hypotheses (and allows for a larger set $S$), but gives much bigger and non-uniform bounds.

\begin{theorem}
\label{bakerthm}
Let $(A,\lambda)$ be a principally polarised abelian surface defined over a number field $K$, with full 2-torsion defined over $K$. Consider the set $S$  of places $v$ of $K$ which are infinite or such that the semistable reduction of $A$ modulo $v$ is not isomorphic to the jacobian of a genus 2 curve. 

Assume also that there is no field extension $L/K$ of degree $[L:K]\leq 4$ such that $(A,\lambda)$ is isogenous to a product of elliptic curves by an isogeny defined over $L$ with kernel $(\Z/2\Z)^2$.

If $|S|<10$, we have a bound 
\[
h_{\mathcal{F}}(A) \leq 10^{66} h_K R_{S} P_S \log^*(h_K R_S  P_S) 
\]
where $\log^*(x) = \max(\log(x),1)$, $h_{\mathcal{F}}$ is the stable Faltings height, $h_K$ is the class number of $K$, $R_S$ is the regulator of $\Ocal_{K,S}^*$ and $P_S$ is the largest norm of a prime ideal in $S$ (1 if there is none).
\end{theorem}

\begin{remark}
First, let us point out that similar effective bounds on the height of such $S$-integral points depending on $S$ exist in \cite{vonKanel14}, as part of the effective Shafarevich conjecture for hyperelliptic curves. Based on studying Weierstrass models with minimality properties, they give the same type of bounds for any $K$ and $S$ (and with bounds of comparable orders), but under the stronger hypothesis that the variety $A$ has good reduction outside $S$ on the base field. Here, we only assume \emph{potentially} good reduction outside $S$: in fact, a pair $(A,\lambda)$ satisfying the hypotheses of the theorem will only obtain good reduction after a quadratic extension in general, in the same manner as one needs a further quadratic extension for an elliptic curve $E$ over $K$ with $E[2] \subset E(K)$ to build an isomorphism towards a Legendre form and ensure semistable reduction \cite[Proofs of Proposition III.1.7 and VII.5.5]{SilvermanAEC}. 

We also note that it is possible to combine the preparatory tools exposed in Sections \ref{secsetup} and \ref{secRunge} to remove the need for using Runge's method and Baker's method on $A_2(2)$, using e.g. the Rosenhain normal form for hyperelliptic curves. Indeed, one can use the cross-ratio coefficients of this normal form to obtain a morphism $A_2(2) \backslash D \rightarrow (\P^1 \backslash \{0,1,\infty\})^3$, sending $S$-integral points to 
$S'$ integral points where $S' = S \cup \{v, v | 2 \}$. We can then use results for the three $S'$-unit equations and bound the height of the solutions such as in \cite[Theorem 1.4]{LeFourn5}. We did not pursue this approach, because although it is more general (estimates would be obtained for any $S$), the preparatory work (going from nice reduction of abelian surfaces to integrality of theta coefficients) is not really simpler, and it does not illustrate the more conceptual application of higher-dimensional Runge and Baker methods that we use here.

\end{remark}
\begin{remark}
The extra condition that $(A,\lambda)$ is not isogenous to a product of elliptic curves by an isogeny with kernel $(\Z/2\Z)^2$ is a byproduct of our use of Baker's method. When $|S|<7$, this condition can be weakened; see Remark \ref{remarkexclusionsets}.
\end{remark}

Theorem \ref{bakerthm} relies on an improved application of Levin's higher-dimensional Baker method, as suggested in \cite{LeFourn5}. The novelty of our approach is that we enhance Levin's ideas for $A_2(2)$ by taking crucial advantage of the behaviour of intersections of the irreducible components of the divisor $D$ with respect to which the integrality is defined.  A \enquote{regular} application of Levin's method would only allow for $|S|<2$, which is already treated by Runge's method above. What happens here is that the intersection of two irreducible components of $D$ can always be blown down to a point by a good rational function on $A_2(2)$. As a consequence, it will be applied to any $S$ for which $|S|<10$ (the number of irreducible components of $D$). This is, admittedly, a peculiar situation, but the authors' hope is that it can be observed in other varieties of interest, and thus improve significantly upon the potential of Levin's ideas.\\

The structure of the paper is as follows. In Section \ref{secsetup}, we start with basic facts and definitions about the variety $A_2(2)$ and its compactification $A_2(2)^S$. The theta coordinates and equations of $A_2(2)^S$ in $\P^9$ and the divisor $D$ are recalled, after which we state precisely the integrality hypothesis and its interpretation in theta coordinates. We prove several useful \enquote{transitivity} statements for the natural action of $\Sp_4(\F_2)$ on $A_2(2)^S$. To conclude this section, we compute the \enquote{graph of intersection} of the irreducible components of $D$, which will be used for both methods afterwards (and helps determining their scope in a rather visual way).

In Section \ref{secRunge} , we make the precise estimates necessary for realising effectively our Runge-type method and regroup them to obtain Theorem \ref{rungethm} and Corollary \ref{cortwoplaces}.

Finally, in Section \ref{secBaker}, after recalling the original strategy of Levin for higher-dimensional Baker's method, we explain this phenomenon of blowing-down cycles and how it applies to the method (in a general context), before making this explicit in the case of $A_2(2)^S$. As mentioned, some closed subsets called \enquote{exclusion sets} appear (which are responsible for the hypothesis of non-isogenicity in Theorem \ref{bakerthm}), which we pin down precisely in the last paragraph of this paper. 

We have chosen to verify some simple computations using \texttt{Magma}. The code for these can be found at 
\[
\texttt{\href{https://github.com/joshabox/IntegralpointsonA22}{https://github.com/joshabox/IntegralpointsonA22}}\; .
\]
\section{Setup of the integrality problem}
\label{secsetup}
In this section, we recall the definitions of our objects of interest and establish their basic properties.

\subsection{\texorpdfstring{The Siegel modular variety $A_2(2)^S$}{The Siegel modular variety A2(2)}}
We start with $\Gamma=\Sp_4(\Z)$ and its action on the Siegel half space 
\[
\Hcal_2 := \{ \tau \in M_2(\C) \, , {}^t \tau = \tau, \, \Im \tau >0 \},
\]
where the positivity is as a real symmetric matrix. It acts naturally by 
\[
\begin{pmatrix} A & B \\ C & D \end{pmatrix} \cdot \tau := (A \tau + B)(C \tau + D)^{-1},
\vspace{-0.4cm}
\]
where $A,B,C,D$ are $2 \times 2$ matrices. For any $M= \begin{pmatrix} A & B \\ C & D \end{pmatrix}$, we define $j_M(\tau) := \det (C \tau + D)$, which is a cocycle for this action (see \cite[Proposition VII.1.1]{Debarre99} for proofs of those claims).

We consider 
\[
\Gamma(2) := \left\{ \gamma \in \Gamma, \gamma = I \mod 2 \right\},
\]
the congruence subgroup of level 2, and denote by $A_2(2)(\C)$ the quotient $\Gamma(2) \backslash \Hcal_2$. The following can be found in \cite[Theorem V.2.5]{ChaiFaltings} and its associated sections. The set $A_2(2)(\C)$ is canonically the set of complex points of a quasi-projective normal algebraic variety of dimension 3 over $\Q$  denoted by $A_2(2)$, which admits a Satake compactification (as a projective normal variety) denoted by $A_2(2)^{S}$. Furthermore, the boundary $\partial A_2(2) := A_2(2)^S \backslash A_2(2)$ is of dimension 1.

The variety $A_2(2)$ is also the coarse moduli space of principally polarised abelian surfaces in characteristic 0 with full symplectic level 2 structure.

\subsection{The integrality question and goal of the paper}

As we will be able to prove later (see Theorem \ref{thmjacorprod} or  \cite[Proposition 7.9 and before]{LeFourn3}), there are ten irreducible effective divisors $D_1, \ldots, D_{10}$ on $A_2(2)^S$ (all defined over $\Q$) whose points of the union 
\[
D = \bigcup_{i=1}^{10} D_i
\]
parametrise (outside the boundary) products of elliptic curves with their natural polarisations (and any choice of symplectic basis).

We will thus be interested, for a number field $K$ and a finite set of places $S$ containing the archimedean ones, in 
\[
(A_2(2)^S \backslash D)(\Ocal_{K,S}),
\]
which corresponds to the set of moduli of triples $(A,\lambda,\alpha_2)$ defined over $K$ whose semistable reduction outside $S$ is an abelian surface not isomorphic (with polarisations) to a product of elliptic curves.

We will see that up to small error for places above 2, this has a natural interpretation in terms of integral points on a model of $A_2(2)$.

Our goal is to bound explicitly those integral points in terms of their Faltings height (and projective height), assuming that $|S|$ is small. First, for $|S| \leq 2$, we will apply Runge's method. Then for $|S| < 10$ we will apply Baker's method. This is more general, but does give worse bounds.

In order to apply Runge's method, it will be convenient to work out precisely the graph of intersection of the divisors, as defined in \cite{LeFourn5}. For this, it is very worthwhile to closely study the action of $\mathrm{Sp}_4(\F_2)$ on $A_2(2)$, and in particular how it permutes the divisors $D_1,\ldots, D_{10}$. 

\subsection{Theta constants and equations of the variety}

The general definition of theta functions for us, inspired by Igusa, is the following: for any $\underline{m}=(m',m'') \in \Z^4$ with $m', m'' \in \Z^2$ (all row vectors), we define for all $\tau \in \Hcal_2$
\[
\Theta_{\underline{m}}(\tau)   :=  \sum_{p \in \Z^2} \exp \left( i \pi (p+m'/2)\tau{}^t(p+m'/2) + i \pi (p+m'/2){}^t m''\right)
\]
For any row vector $\underline{n} \in \Z^4$, it is easily checked that
\[
\Theta_{\underline{m}+2\underline{n}} = (-1)^{m'{}^t n''} \Theta_{\underline{m}} \quad \textrm{and} \quad \Theta_{-\um} (\tau) = \Theta_{\um} (\tau).
\]


This already proves that when $m'{}^t{m''}$ is odd, the associated theta function is 0 at $z=0$. On another hand, because of this formula, the square of $\Theta_{\um}$ only depends on $\um$ modulo $2 \Z^4$. We have thus defined 16 functions of $\tau$. Six of those are zero, and the ten remaining ones correspond to the classes of $\um$ modulo 2 that are called the \textit{even theta characteristics}, listed here: 
\begin{equation}
\small
\label{eqevenchartheta}
E = \{(0000),(0001),(0010),(0011),(0100),(0110),(1000),(1001),(1100),(1111) \}.
\end{equation}

Recall  \cite[Theorem 5.2]{vdG82} that the ten even theta functions define an embedding
\begin{equation}
\label{eqdefplongementpsi}
\fonction{\psi}{A_2(2)}{\P^9}{\overline{\tau}}{(x_\um=\Theta_{\underline{m}}^4 (\tau))_{\underline{m} \in E}}
\end{equation}
which induces an isomorphism between $A_2(2)^S$ and the subvariety of $\P^9$ (with coordinates indexed by $E$) defined by the linear equations 
\begin{eqnarray}
\label{eqA22inP9}
x_{1000} - x_{1100} + x_{1111} - x_{1001} & = & 0 \\
\label{eqA22inP92}
x_{0000} - x_{0001} - x_{0110} - x_{1100} & = & 0 \\
\label{eqA22inP93}
x_{0110} - x_{0010} - x_{1111} + x_{0011} & = & 0 \\
\label{eqA22inP94}
x_{0100} - x_{0000} + x_{1001} + x_{0011} & = & 0 \\
\label{eqA22inP95}
x_{0100} - x_{1000} + x_{0001} - x_{0010} & = & 0 
\end{eqnarray}
together with the quartic equation 
\begin{equation}
\label{eqA22inP96}
\left( \sum_{m \in E} x_m^2 \right)^2 - 4 \sum_{m \in E} x_m^4 = 0.
\end{equation}
\begin{remark}
Following \cite[p. 396 and 397]{Igusa64}, these equations can also be presented in a reduced form as defining a quartic in $\P^4$, for the reader who would prefer doing computations manually (which we do not, with the exception of the proof of Corollary \ref{cortwoplaces}). Our choice has been to use \texttt{Magma} throughout to keep pure discussion of the computations to the minimum (and ensure correctness).
\end{remark}

Now, the theta functions enjoy a modularity transformation formula \cite[p. 227]{Igusa64}, whose expression is the following: 
\[
\Theta_{\um}(M \cdot \tau) = \zeta (M) e (\phi_{\um}(M^{-1})) \sqrt{j_M(\tau)}\Theta_{\um \odot M}  (\tau) ,
\]
where $\zeta (M)$ is a 8-th root of unit depending only on $M$, $\phi_{\um}$ will be made explicit later, and 
\begin{equation}
\label{eqactSp4onE}
\um \odot \begin{pmatrix} A & B \\ C & D \end{pmatrix} := \um \cdot  \begin{pmatrix} A & B \\ C & D \end{pmatrix} - (({}^t CA)_0,   {}^t DB)_0),
\end{equation}
where $A_0$ is the row vector formed by the diagonal coefficients of a square matrix $A$. We only care about the fourth powers of our theta functions, and as $4\phi_\um(M^{-1})$ can be made more explicit: we obtain
\begin{equation}
\label{eqmodularityxm}
x_\um (M \cdot \tau) = \zeta(M)^4 (-1)^{\um \cdot {}^t((B{}^t A)_0), (C{}^t D)_0))}j_M(\tau)^2 x_{\um \odot M}( \tau).
\end{equation}

The action of $\Gamma$ on $\Hcal_2$ thus amounts via $\psi$ to an action of $\Sp_4(\F_2)$ on $A_2(2)^S \subset \P^9$ via signed permutations of coordinates, in the shape 
\begin{equation}
\label{eqdotactioncoordinates}
(x_\um)_{\gm \in E}^M = \left((-1)^{\um \cdot {}^t((B{}^t A)_0), (C{}^t D)_0))} x_{\um \odot M}\right)_{\gm \in E}. 
\end{equation}

Forgetting about the signs, the permutation of coordinates is induced by the action on $E \subset \F_2^4$ in \eqref{eqactSp4onE}.

This action is the key to understanding the combinatorics at play, so we first explain all its basic properties.

\subsection{Properties of the dot action}

 We identify $E$ with its natural image inside $\F_2^4$. 

\begin{lemma}
The operation $(\underline{m},M) \mapsto \underline{m} \odot M$ is indeed a group action of $\operatorname{Sp}_4(\F_2)$ on $\F_2^4$, which furthermore preserves the quadratic form $q_2$ on $\F_2^4$ given by 
\[
q_2((m',m'')) = (m') {}^t m''.
\]
Consequently, this action stabilises $E$, as it is the set of isotropic vectors of $q_2$.
\end{lemma}

\begin{proof}
This is the content of \cite[Propositions V.6.1 and V.6.3]{IgusaThetaF}, where $q_2$ is denoted by $e$ and $m'$ and $m''$ are normalised with half-integer values. The curious reader can check it by hand using the definition of $\Sp_4(\F_2)$.
\end{proof}

We can now study more finely this action, with definitions borrowed from Igusa again.

\begin{definition}
\label{defsyzygousazygous}
For $x,y,z \in \F_2^4$, we define
\[
e(x,y,z) = q_2(x) + q_2(y) + q_2(z) + q_2(x+y+z).
\]
A triple of distinct $x,y,z \in E$ is then called \emph{syzygous} if $e(x,y,z) = 0$ (i.e. if $x+y+z \in E$) and \emph{azygous} otherwise. 

A quadruple of distinct $x,y,z,t \in E$ is a \emph{Göpel quadruple} if every triple in it is syzygous, and an \emph{azygous quadruple} if every triple in it is azygous.
\end{definition}

\begin{proposition}
\label{proptrans}
The $\odot$ action of $\Sp_4(\F_2)$ restricted to $E$ has the following properties: 
\begin{itemize}
    \item[$(a)$] It is 2-transitive.
\item[$(b)$]It acts transitively on the 60 syzygous triples of $E$ (and also on the 60 azygous triples of $E$).
\item[$(c)$]It acts transitively on the 15 Göpel quadruples of $E$ (and hence on their complements in $E$), and the 15 azygous quadruples of $E$.
\end{itemize}
\end{proposition}
\begin{proof}

First, the counting of triples and quadruples with the required properties can be done by hand or via \texttt{Magma}.

With different notations and wording, \cite[Proposition V.6.2]{IgusaThetaF} states the following: for any two sequences $(x_1, \cdots, x_k)$ and $(y_1, \cdots, y_k)$ of $\F_2^4$, there exists $M \in \Sp_4(\F_2)$ such that $x_i \odot M = y_i$ if and only if the subsequences which are affinely dependent have the same indices, and for any triples of distinct indices $i,j,k$, $q_2(x_i) = q_2(y_i)$ (and the same for $j,k$) and $e(x_i,x_j,x_k) = e(y_i,y_j,y_k)$.

For sequences of elements of $E$ (i.e isotropic vectors for $q$), for $k=2$, it gives part $(a)$ and for $k=3$, it gives part $(b)$ (affine independence is automatic for distinct triples in $\F_2^4$, as affine lines have only two elements).

Finally, the Göpel quadruples can be checked to be specific translates of maximal totally isotropic subpaces of $\F_2^4$ by elements of $E$; in particular they are automatically not affinely independent and we can use the $k=4$ case of the property.
\end{proof}

We now explain how much a given subset of $E$ can be expanded to one of those above.

\begin{lemma}
\label{lemcompletesubsetsE}
\hspace*{\fill}
\begin{itemize}

    \item [$(a)$] Any pair of distinct $x,y \in E$ can be completed into 4 syzygous triples, and 4 azygous triples.
    \item[$(b)$] Any syzygous triple can be completed into a unique Göpel quadruple, and any azygous triple can be completed into a unique azygous quadruple.
    \item[$(c)$] A syzygous triple is disjoint with exactly two Göpel quadruples, and an azygous quadruple is disjoint with exactly three Göpel quadruples.
    %
    %
    \item[$(d)$] No two Göpel quadruples are disjoint.
\end{itemize}
\end{lemma}

\begin{proof}
Item $(a)$ is simply using that $z \mt e(x,y,z)$ is linear on $\F_2^4$ and non-zero when $x \neq y$. The structure of $E$ as set of isotropic vectors then imposes that there are as much elements of $E$ with image 0 (completing $\{x,y\}$ to a syzygous triple) as there are with image 1 (completing $\{x,y\}$ to an azygous triple).

Items $(b)$, $(c)$ and $(d)$ can be obtained by looking at a fixed triple or quadruple and using transitivity.
\end{proof}

\begin{remark}
For a quick exploration of those triples and quadruples, the last section gives a list of all Göpel and azygous quadruples. The \texttt{Magma} code also verifies all the claims above.
\end{remark}

\subsection{Theta coordinates, type of abelian surface, and semistable reduction}

First, the vanishing of theta coordinates indicates if the abelian surface is a jacobian or not. More precisely, we have the following.

 \begin{theorem}
 \label{thmjacorprod}
 Let $P =(A,\lambda,\alpha_2)$ a principally polarised abelian variety defined over a number field $K$ together with a symplectic 2-torsion basis such that the corresponding point $\psi(P) \in \P^9$ has coordinates in $K$. Then, $A[2](K)=A[2]$ and: 
 
 \begin{itemize}
     \item If no coordinate is 0, there exists a curve $C$ defined over $K$ and of genus 2 such that $\operatorname{Jac}(C)$ is isomorphic to $(A,\lambda)$ over an extension $K'/K$ of degree 2, and its six Weierstrass points are $K$-rational.
     
     \item Otherwise, exactly one coordinate is 0 and there exist two elliptic curves $E_1,E_2$ defined over $K$ such that $(A,\lambda)$ is isomorphic to $E_1 \times E_2$ over an extension $K'/K$ of degree 4.
 \end{itemize}
 \end{theorem}
 
 \begin{proof}
 \hspace*{\fill}
 
 $\bullet$ The first case relies mainly on Thomae's formulae \cite[Chapter 6]{MumfordTataII} expressing cross-ratios of fourth powers of theta constants of a jacobian in terms of the roots of the sextic defining a curve. This allows in turn to rebuild normal forms from those cross-ratios, for example with the Rosenhain normal form 
 \[
 C: \quad y^2 = x (x-1) (x - \lambda_1)(x-\lambda_2)(x-\lambda_3)
 \]
 where 
 \[
 \lambda_1 = \frac{\Theta_{0000}^2 \Theta_{0010}^2}{\Theta_{0001}^2 \Theta_{0011}^2}, \quad , \lambda_2 = \frac{\Theta_{0010}^2 \Theta_{1100}^2}{\Theta_{0001}^2 \Theta_{1111}^2}, \quad \lambda_3 = \frac{\Theta_{0000}^2 \Theta_{1100}^2}{\Theta_{0011}^2 \Theta_{1111}^2}
 \]
 \cite[Lemma 2.5 with notations of equation (2.5)]{ClingherMalmandier20}. These only express the parameters in terms of squares of theta constants, but in fact, using classical relations between them (or the equations of $A_2(2)$), one obtains that 
 \[
 \lambda_1 = \pm \frac{x_{1000} x_{1001} - x_{0000} x_{0001} - x_{0010} x_{0011}}{2 x_{0001} x_{0011}}.
 \]
 The sign could be determined if necessary by complex analysis, and similar equations hold for $\lambda_2$ and $\lambda_3$. We thus obtain in this case a curve $C$ defined over $K$ such that $\operatorname{Jac}(C) \cong (A,\lambda)$ (over $\overline{K}$) and $\operatorname{Jac}(C)[2]$ is fully defined over $K$ (as the Weierstrass points of $C$, $\{0,1,\infty,\lambda_1,\lambda_2,\lambda_3\}$ are). After a good permutation of the Weierstrass points (which amounts exactly to the dot action), we  ensure that there is a basis $\beta_2$ of the two-torsion on $\operatorname{Jac}(C)$ such that $\psi(P) = \psi((\operatorname{Jac}(C),\beta_2))$, which implies that the two triples are isomorphic over $\overline{K}$.

 Now, an automorphism of $\operatorname{Jac}(C)$ fixing the polarisation and the full 2-torsion comes by Torelli's theorem from an automorphism of $C$ fixing pointwise the Weierstrass points, and such an automorphism is necessary trivial or the hyperelliptic involution, so by descent $(A,\lambda,\alpha_2)$ and $(\operatorname{Jac}(C),\beta_2)$ are isomorphic over a quadratic extension $K'$ of $K$.
 
 $\bullet$ Assume now that at least one coordinate is 0. By \cite{OortUeno}, one then knows that over $\Kb$, $(A,\lambda)$ is isomorphic to a product of elliptic curves with the product polarisation. After a permutation by an element of $\Sp_4(\Z)$, one can thus assume that $(A,\lambda,\alpha_2)$ is represented in $\Hcal_2$ by a diagonal matrix $\tau = \begin{pmatrix} \tau_1 & 0 \\ 0 & \tau_2 \end{pmatrix}$. Then the coordinates split because
 \[
 \Theta_{\um} (\tau) = \Theta_{m'_1 m''_1}(\tau_1) \Theta_{m'_2 m''_2}(\tau_2),
 \]
 where 
 \[
 \Theta_{ab}(\tau_1) = \sum_{n \in \Z} \exp ( i \pi (n+a/2)^2 \tau_1 + i \pi (n+a/2)b),
 \]
 so we fall back to four possible one-dimensional theta functions. One of them ($\Theta_{11}$) is always 0, and the other three do not vanish on the Poincar\'e half plane. Apart from the coordinate $(1111)$, we are thus looking (up to permutation of coordinates) at the Segre embedding $\P^2 \times \P^2 \rightarrow \P^8$, where in each $\P^2$ the coordinates are the three fourth powers of non-zero theta constants for respectively $\tau_1$ and $\tau_2$.
 
 We can thus assume (after renormalisation) that each $\Theta_{ab} (\tau_i)^4$ ($ab\in \{ (00),(01), (10) \}$, $i \in \{1,2\}$) belongs to $K$. Now, the $j$-invariant of the elliptic curve associated to $\tau_1$ is a rational function of the three fourth powers of theta constants with rational coefficients \cite[p. 29]{BruiniervdGZagier}, so we can find elliptic curves $E_1$ and $E_2$ defined over $K$ and such that $(A,\lambda) \cong E_1 \times E_2$ over $\Kb$ with the product polarisation. For similar reasons as in the first case, it thus amounts to looking at the automorphisms of $E_1 \times E_2$ preserving pointwise the 2-torsion, and there are always exactly 4 of them (possible extra automorphisms of elliptic curves do not preserve the 2-torsion pointwise), given by $\pm \Id$ on each component. 
 
 It thus means that there exists an extension $K'/K$ of degree 4  such that $(A,\lambda,\alpha_2)$ is isomorphic over $K'$ to $E_1 \times E_2$.
 \end{proof}

Now, this characterisation of the type of abelian surface of theta coordinates extends to every field of characteristic $\neq 2$, because theta constants can be intrinsically defined as algebraic theta constants (by Mumford's theory of theta functions), and are compatible with reduction outside of characteristic 2. This leads to the following result (see \cite[Proposition 8.4]{LeFourn4} which also deals with the case of reduction to a product of elliptic curves).

\begin{proposition}
\label{propalgthetafoncetreduchorsde2}
Let $K$ be a number field and $\gP$ a maximal ideal of $\Ocal_K$ of residue field $k(\gP)$ with $\operatorname{char} k(\gP) \neq 2$.
	Let $P = \overline{(A, \lambda,\alpha_2)} \in A_2(2)(K)$. Then, $\psi(P) \in \P^9(K)$ and if the semistable reduction of $A$ modulo $\gP$ is a jacobian of hyperelliptic curve, the reduction of $\psi(P)$ modulo $\gP$ has no zero coordinate, in other words every coordinate of $\psi(P)$ has the same $\gP$-adic norm.
	

\end{proposition}



For the places above 2, the situation is a bit more complicated (in part because there is no good theory of algebraic theta constants in characteristic two), but using Igusa invariants, we can say the following.

\begin{proposition}
\label{propalgthetachar2}
Let $K$ be a number field and $\gP$ a maximal ideal of $\Ocal_K$ above 2, and $P = \overline{(A, \lambda,\alpha_2)} \in A_2(2)(K)$ as in the previous proposition.
\begin{itemize}
\item[$(a)$] If the semistable reduction of $A$ modulo $\gP$ is a jacobian of hyperelliptic curve, the coordinates of the reduction of $\psi(P)$ modulo $\gP$ all satisfy 
\[
|x_\um|_\gP \geq |2|_\gP^6 \cdot \max_{\um' \in E} |x_{\um'}|_\gP
\]
%
%
\item[$(b)$]In all cases, the coordinates satisfy 
\[
|x_\um|_\gP = \max_{\um' \in E} |x_{\um'}|_\gP
\]
with at most 6 exceptions $\um \in E$. 
\end{itemize}
\end{proposition}
\begin{remark}
\label{rem6vancoord}
An important point is that, although crude, part $(b)$ only relies on the explicit equations in $\P^9$. Consequently, we can and will use it as a go-to estimate every time one does not have any better option.
\end{remark}

\begin{proof}
Part $(a)$ is in \cite[Proposition 8.7]{LeFourn3}. For $(b)$, suppose this is not true. After normalisation to coordinates in $\Ocal_{K,\gP}$, we can assume that at least one of them is invertible and seven have positive valuation. We consider $A_2(2)$ as a scheme over $\Z$ and suppose that 7 coordinates vanish. These 7 coordinates have indices ranging over the complement of a syzygous or an azygous triple, so it suffices to look at one explicit complement of a syzygous triple and one complement of an azygous triple. We first suppose that the variables $x_{\um}$ for $\um$ in the complement of the syzygous triple $\{(1001),(0100),(1111)\}$ all vanish. Then the equations between the $x_{\um}$s directly imply that $x_{\um}=0$ for all $\um \in E$ (and this holds over $\Z$), a contradiction. Next, we consider the complement of the azygous triple $\{(0000),(0100),(0001)\}$. Now the equations imply that the ideal generated by these seven $x_{\um}$s contains
\[
x_1+x_5,x_2+x_5,x_3,x_4,2x_5,x_5^4,x_6,x_7,x_8,x_9,x_{10}.
\]
Here $x_i$ is $x_{\um}$ where $\um$ is the $i$th element of $E$ as displayed in (\ref{eqevenchartheta}). In particular, the radical of the ideal contains all coordinates, so all of them vanish, also a contradiction. Finally, we base change this argument to the residue field of $\Ocal_K$ at $\gP$.
%
\end{proof}

\begin{remark}
When we find 6 simultaneously vanishing coordinates modulo $\gP$, their indices are not random: they will turn out to form the complement subset to a Göpel quadruple (Definition \ref{defsyzygousazygous}), as proven in the next section.

\end{remark}

\subsection{The graph of intersection of the divisors}

We can now figure out precisely what is the graph of intersection of our divisors. Recall that it is defined as follows, following \cite[Section 5]{LeFourn5}.

\begin{definition}
The vertices of the \emph{graph of intersection} are non-empty set-theoretic intersections
\[
Z_I:= \bigcap_{i \in I} D_i (\overline{K}) \text{ for } I\subset E.
\]
A set of indices $I$ is called \emph{optimal} if there is no set $J\supsetneq I$ such that $Z_J=Z_I$. The \emph{depth} of a vertex $Z$ is defined to be the size of an optimal set $I$ such that $Z=Z_I$. An edge goes from $Z_I$ to $Z_J$ if $Z_J \subsetneq Z_I$ (equivalently, if $I$ and $J$ are optimal and $I \subsetneq J$) with no intermediary intersection. Finally, the \emph{cone of ancestors} of a vertex $Z$ is the set of vertices $Z'$ from which starts a path leading to $Z$.
\end{definition}
In our situation, with help of the equations, we obtain the following graph of intersection. The number in each oval is the number of vertices of a given depth (and to which type of optimal subset of $E$ they correspond). The number in each thick arrow corresponds, for each vertex above, to its number of children below.

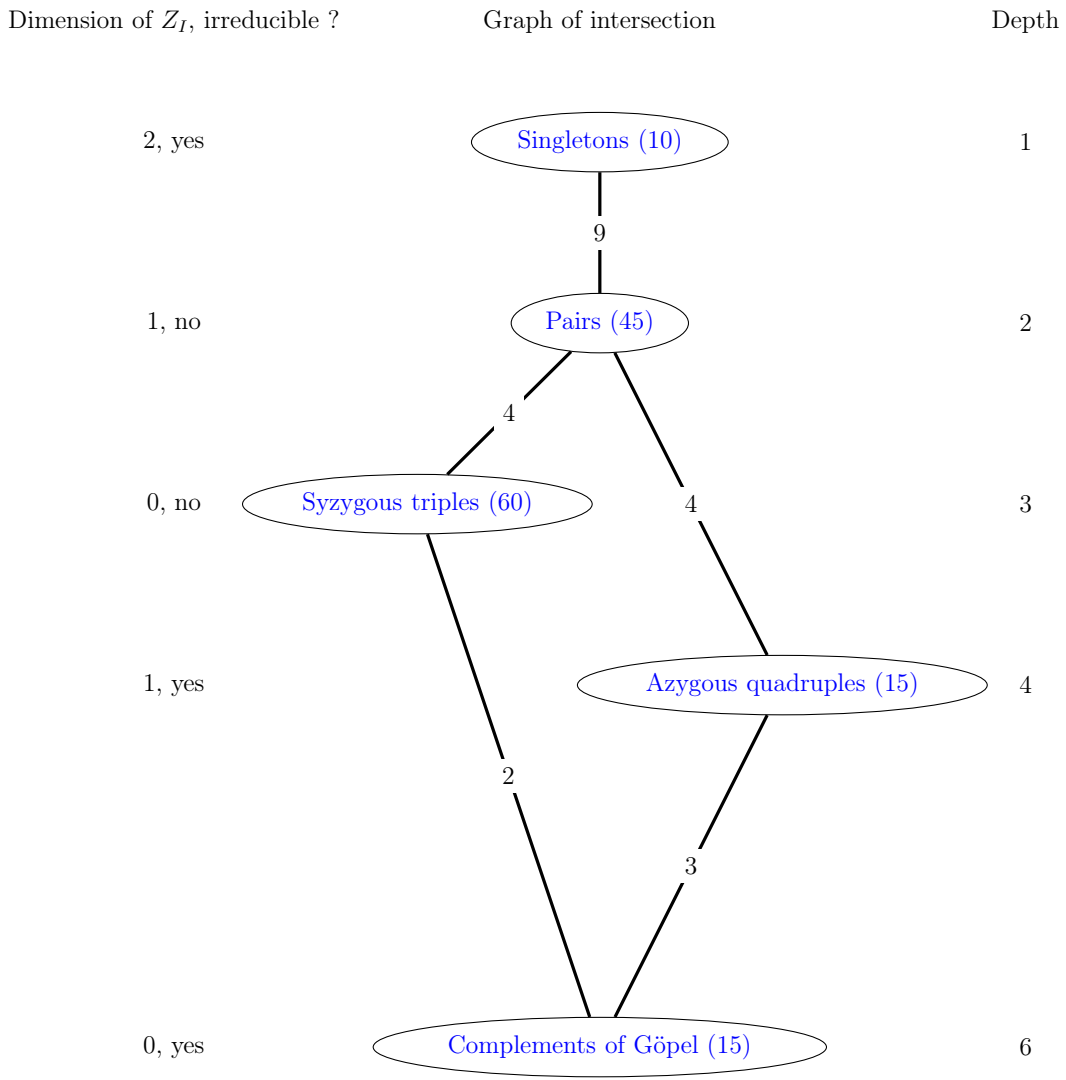
\begin{center}
\begin{tikzpicture}[scale=0.8, every node/.style={scale=0.8}]
\tikzstyle{oval}=[ellipse,draw,text=blue];

\node (GI) at (0,8) {Graph of intersection};
\node (De) at (7,8) {Depth};
\node (D1) at (7,6) {1};
\node (D2) at (7,3) {2};
\node (D3) at (7,0) {3};
\node (D4) at (7,-3) {4};
\node (D6) at (7,-9) {6};
\node (Z) at (-7,8) {Dimension of $Z_I$, irreducible ?};
\node (Z1) at (-7,6) {2, yes};
\node (Z2) at (-7,3) {1, no};
\node (Z3) at (-7,0) {0, no};
\node (Z4) at (-7,-3) {1, yes};
\node (Z6) at (-7,-9) {0, yes};
\node[oval] (S) at (0,6) {Singletons (10)};
\node[oval] (P) at (0,3) {Pairs (45)};
\node[oval] (Sy) at (-3,0) {Syzygous triples (60)};
\node[oval] (Az) at (3,-3) {Azygous quadruples (15)};
\node[oval] (C) at (0,-9) {Complements of Göpel (15)};
\draw[very thick] (S)--(P) node[midway,fill=white] {9};
\draw[very thick] (P)--(Sy) node[midway,fill=white]{4};
\draw[very thick] (P)--(Az) node[midway,fill=white]{4};
\draw[very thick] (Sy)--(C) node[midway,fill=white]{2};
\draw[very thick] (Az)--(C) node[midway,fill=white]{3};
\end{tikzpicture}
\end{center}

\noindent {\bf Proof of graph of intersection.}
To build the graph of intersection, we start with singletons and then add elements step by step.

First, note that transitivity of the dot action of $\Sp_4(\F_2)$ on each of the subsets of $E$ displayed in the ovals (and therefore on the corresponding sets of vanishing coordinates because of \eqref{eqdotactioncoordinates}), detailed in Proposition \ref{proptrans}, allows us to reduce to a single singleton, pair, syzygous triple, azygous quadruple or complement of G\"opel quadruple and thus saves us a lot of work.

At each step, we determine whether $Z_I$ is optimal. To prove that a set $I$ is not optimal, we formally manipulate the given equations for $A_2(2)$ together with $x_i=0$ for $i\in I$ to obtain that $x_j=0$ for some $j\notin I$. Similarly, we extract the dimension and number of irreducible components from the explicit equations. 

Conversely, to prove that a set $I$ is optimal, we exhibit for each $j\notin I$ a point $P_j\in Z_I$ such that $P_j$ has non-zero $j$th coordinate. Such points can always be found as \emph{deepest points}: for each complement $I$ of a G\"opel quadruple there is a unique point $P\in A_2(2)$, all of whose coordinates are in $\{0,1,-1\}$, such that $x_i(P)=0$ if and only if $i\in I$.

Finally, we use Lemma \ref{lemcompletesubsetsE} to determine the number of children displayed in the arrows. 

This process is rather laborious and error-prone to do by hand, so we have implemented it in \texttt{Magma}.

\begin{remark}
By formal computations, one can notice that this process  would give \emph{exactly} the same result for these equations over any base field (finite or not) of characteristic unequal to 2 and 3. In particular, reductions of divisors do not intersect more than the divisors over $\Q$ (but to be precise, the scheme-theoretic intersections are sometimes not reduced) except over $\F_2$ and $\F_3$. We will not need this, but have nonetheless worked it out in the \texttt{Magma} file.
This phenomenon is implicit in the bounds of Proposition \ref{explicitprop2} below and its proof.
\end{remark}

\begin{remark}
Even though we have not undertaken the verification of this claim and we do not need it later, it is likely that the 5 types of optimal subsets are closely related to the moduli interpretations of points of $A_2(2)^S$ (as jacobians of stable curves of genus 2). More precisely, following the notations of \cite[Proposition 1 ]{NamikawaUeno73} and taking into account that their compactification is a blow-up of the Satake compactification, we can expect that singletons are given by products of elliptic curves ($\Ncal$, type II), pairs by elliptic curves ($\Bcal$, type  III), syzygous triples and azygous quadruples by one or two rational curves ($\Ccal$, type IV) and complements of Göpel quadruples by two rational curves meeting at three points ($\Dcal$, type V).
\end{remark}
\section{Runge's method refined}
\label{secRunge}

We apply the formalism of the graph of intersection to Runge's method. Then \cite[Proposition 5.5]{LeFourn5} tells us that if $|S| \leq 2$, one can obtain an explicit bound on the height of points in $(A_2(2)^S \backslash D)(\Ocal_{K,S})$. Let us explain why.

The shape of the graph of intersection tells us the following: no union of two cones of ancestors recovers all the graph in depth 1 (i.e. contains all the divisors). This is due to the fact that any two Göpel quadruples have non-empty intersection, so the union of two complements of Göpel quadruples cannot be the full set $E$.

To exploit this fact, we need to define a consistent notion of $v$-closeness to the intersections of divisors $Z_I$ (for every place $v \in M_K$). For example, if $K$ is a number field and $v$ is a finite place of $K$ not above 2 or 3, then $P\in A_2(2)(K)$ will be $v$-close to $Z_I$ if and only if the reduction of $P$ mod $v$ belongs to the Zariski closure of $Z_I$.

As a consequence, for any integral point $P \in (A_2(2)^S \backslash D) (\Ocal_{K,S})$ where $|S|=2$, there are at most two places $v \in M_K$ for which $P$ is $v$-close to one of the divisors, and thus generates a cone of ancestors.
Taking away those two cones of ancestors, there remains a divisor $D_i$ which is $v$-far from $P$ for all $v \in M_K$, and thus allows to bound the local heights $h_{D_i,v}(P)$ for all $v$, and therefore the global height $h_{D_i}(P)$. This will be particularly easy to do here as the $D_i$ are given as coordinate hyperplanes. This is how we obtain an absolute bound on the height of $\psi(P)$.

To obtain such a bound in practice, more refined estimates are needed for three different reasons:
\begin{itemize}
    \item Our definition of integral points comes from the moduli space structure (and not the explicit equations), which makes a slight difference in the bounds.
    \item The graph of intersection of the divisors is different over fields of characteristic 2 or 3 (which tells us that even though $Z_I$s are distinct, they might still be close enough to need a finer definition of closeness to distinguish them).
    \item We need to evaluate closeness in the archimedean case.
\end{itemize}

\subsection{Estimates on the size of theta functions}
In this subsection, we refine the estimates in \cite[Proposition 8.5]{LeFourn3} on sizes of theta functions. For archimedean places and places above 2 and 3, this will provide a quantitative analogue for the part of the graph of intersection that we will need, while at other finite places it is merely a confirmation of what we already knew.

Instead of analysing the Fourier expansions of the theta functions as was done by Streng in \cite{Strengthesis} and quoted in \cite{LeFourn3}, we  only make use of the six equations satisfied by the fourth powers of the theta functions to obtain our estimates.
\begin{proposition}
\label{explicitprop1}
Consider $\tau \in \mathcal{H}_2$, and suppose that $K$ is a number field such that $x_{\um}\in K$ for each $\um\in E$. Let $|\cdot|$ be any norm on $K$. The set of $\underline{m} \in E$ satisfying
\[
|x_\um|<\begin{cases}\max_{\um' \in E} |x_{\um'}| \;\; &\text{ if } |\cdot| \text{ non-archimedean} \\
 \frac{1}{27}\max_{\um' \in E} |x_{\um'}| \;\; &\text{ if } |\cdot| \text{ archimedean}
\end{cases}
\]
either has size at most 4, or is contained in one of the 15 complements of Göpel quadruples. 
\end{proposition}
\begin{remark}
This constant $1/27=0.037..$ is a slight improvement on the constant $0.42^4=0.031..$ found by Le Fourn \cite{LeFourn3} based on Streng's estimates \cite{Strengthesis}.
\end{remark}

\begin{proof}
In the non-archimedean case, this can be verified explicitly by considering $A_2(2)$ over $\Z$, as in the proof of Proposition \ref{propalgthetachar2} (c). Alternatively one can reason along the lines of the below proof for archimedean norms. We thus assume that $|\cdot|$ is an archimedean norm. We take large rather than small theta functions as our point of view, showing that there are at least four $x\in \{x_{\um}\mid \um\in E\}$ of size $\frac{1}{27}\max_{\um}|x_{\um}|$ and, whenever there are at most five $x$ of size $\geq \frac{1}{27}\max_{\um}|x_{\um}|$ then these contain a G\"opel quadruple.

One of the $x_{\um}$ is the largest, say of size $M:=\max_{\um}|x_{\um}|$. By considering a linear equation featuring $x_{\um}$, we find a second $x$-coordinate of size at least $M/ 3$. By transitivity of $\mathrm{Sp}_4(\F_2)$ on pairs, we may assume this pair of large $x$-coordinates is $\{x_{0000},x_{0010}\}$. In fact, we may suppose that $|x_{0000}|=M$ and $|x_{0010}|\geq M/3$. 

So equations (\ref{eqA22inP92})-(\ref{eqA22inP95}) all contain one $x_{\um}$ of size at least $M/3$, and hence a second $x_{\um}$ of size at least $M/9$. Note that $x_{0110},x_{0100},x_{0001},x_{0011}$ are the four (out of 8 remaining variables) that occur in two of those equations. These are exactly the four variables extending $\{x_{0000},x_{0010}\}$ to a syzygous triple. By transitivity of the action of $\mathrm{Sp}_4(\F_2)$ on syzygous triples, these choices are thus equivalent. 
Let us first assume that all four of those are in absolute value $<M/9$. These four together form a G\"opel quadruple. Then equations (\ref{eqA22inP92})-(\ref{eqA22inP95}) show that $|x_{1100}|,|x_{1111}|,|x_{1001}|,|x_{1000}|\geq M/9$, thus yielding a total of eight $x_{\um}$ of size $|x_{\um}|\geq M/9$.  

We may thus assume that $|x_{0100}|\geq M/9$, giving a syzygous triple of large coordinates. We can use equations (\ref{eqA22inP92}) and (\ref{eqA22inP93}) to find more ``large'' $x_{\um}$. In particular, looking at (\ref{eqA22inP93}) we have three cases: $|x_{0110}|\geq M/9$ (case (i)), $|x_{0011}|\geq M/9$ (case (ii)) and $|x_{1111}|\geq M/9$ (case (iii)). Consider first case (i). If $|x_{0110}|\geq M/9$, then we have found four coordinates of size $M/9$, namely $\{x_{0000},x_{0010},x_{0100},x_{0110}\}$, and we check that this is a G\"opel quadruple: the unique G\"opel quadruple extending our syzygous triple.

Next, consider case (ii), so $|x_{0011}|\geq M/9$. Then in equation (\ref{eqA22inP92}), one of $|x_{0001}|,|x_{0110}|,|x_{1100}|$ is $\geq M/9$. The second of these is the case just treated. The first gives us exactly that the variables in the G\"opel quadruple $\{x_{0000},x_{0001},x_{0010},x_{0011}\}$ are all at least $M/9$ in absolute value. This is the second (out of two in total) G\"opel quadruple containing the pair $\{x_{0000},x_{0010}\}$. In the third case, if $|x_{1100}|\geq M/9$, then from equation (\ref{eqA22inP9}) we find that a sixth coordinate must be at least $M/27$ in absolute value. 

Finally, we consider case (iii), where $|x_{1111}| \geq M/9 $. In equation (\ref{eqA22inP92}), we now again have three possibilities, of which $|x_{0110}|\geq M/9$ has already been treated, and the case $|x_{0001}|\geq M/9$  yields six theta functions of size at least $M/27$ by considering equation (\ref{eqA22inP9}). 

This leaves $|x_{1100}|\geq M/9$. The 5-set of ``large'' $x_{\um}$s we have selected so far does not contain a G\"opel quadruple. Moreover, the linear equations appear to be perfectly happy with the sizes of the variables: each equation has two large variables. It is here that we must invoke the power of the degree 4 equation. Assume that all remaining 5 variables have size strictly smaller than $\eps M$. Then $|x_{1100}-x_{0000}|,|x_{0100}-x_{0000}|\leq 2\eps M$ by equations (\ref{eqA22inP92}) and (\ref{eqA22inP94}) and $|x_{1111}-x_{0000}|,|x_{0010}-x_{0000}|\leq 5\eps M$ by substituting equations (\ref{eqA22inP92}) and (\ref{eqA22inP94}) into equations (\ref{eqA22inP9}) and (\ref{eqA22inP95}) respectively. So in the quartic equation (\ref{eqA22inP96}) we may replace each such $x$ by $x_{0000}$, at the expense of adding a term of size $\leq 5\eps M$ or $\leq 2\eps M$. This gives
\[
5x_{0000}^4+C=0, \text{ where } |C|\leq (6543\eps^4 + 5656\eps^3 + \eps^2 + 56\eps)|x_{0000}|^4
\]
by the triangle inequality. Now $\eps\leq 1/27$ yields $|C|<5|x_{0000}|^4$, a contradiction. 
\end{proof}

\begin{proposition}
\label{explicitprop2}
Let $|\cdot|$ be a non-archimedean norm. Then for each G\"opel quadruple $Q$, one $x\in \{x_{\um} \mid \um \in E\}$ must satisfy
\[
|x| \geq |2||3| \max_{i \in E} |x_i|.
\]
When $|\cdot|$ is archimedean, the same is true with $|2||3|$ replaced by $0.051$. 
\end{proposition}
\begin{remark}These factors of $|2|$ and $|3|$ are strictly necessary. Indeed, the subscheme of $A_2(2)/\F_2$ given by the vanishing of the variables indexed by the G\"opel quadruple $\{(0000),(0001),(0010),(0011)\}$ is zero-dimensional and contains the point $(0:0:0:0:1:1:1:1:1:1)$. Similarly, this scheme over $\F_3$ contains $(0:0:0:0:1:-1:1:-1:1:-1)$. 
\end{remark}

\begin{proof}
Let $|\cdot|$ be a norm. When $|\cdot|$ is archimedean, suppose that $|x_i|<\ep M$ for each $i$ in the G\"opel quadruple $G:=\{x_{0000},x_{0001},x_{0010},x_{0011}\}$, where $M=\max_{\um \in E} |x_{\um}|$ and $\ep\leq 1$. We write $o(z)$ for any complex number of size $|o(z)|<z$. When $|\cdot|$ is non-archimedean, we may assume after scaling that $x_i\in \mathcal{O}_K$ for all $i\in E$ and one $x_i$ equals 1. Now suppose for each $i$ in this G\"opel quadruple that $x_i\equiv 0 \mod \pi^n$, where $n>0$ and $\pi$ is a uniformiser.  From (\ref{eqA22inP92})-(\ref{eqA22inP95}) we deduce for archimedean norms that 
\begin{align}
\label{eqn55}
-x_{1100}=x_{0110}+o(2\ep M), \; x_{0110}=x_{1111}+o(2\ep M)\; 
\end{align}
and 
\begin{align}
\label{eqn6}
-x_{1001}=x_{0100}+o(2\ep M),\; x_{0100}=x_{1000}+o(2\ep M).
\end{align}
Substituting the above into (\ref{eqA22inP9}) yields
\begin{align}
\label{eqn7}
2x_{0100}+2x_{0110}=2o(4\ep M).
\end{align}
 This implies that
\begin{align*}
    x_{1000},x_{1001}=x_{0100}+o(2\ep),\;\; x_{0110}=x_{0100}+o(4\ep),\;\; x_{1100},x_{1111}=x_{0100}+o(6\ep).
\end{align*}
In particular, all six $x_i$ for $i\notin G$ are of similar size. When $\ep$ is sufficiently small, this will contradict the quartic equation (\ref{eqA22inP96}). 

Since $\ep\leq 1$, we must have $|x_i|=M\neq 0$ for some $i\in E\setminus G$. In what follows, the two choices $i=0100$ and $i=0110$ will be equivalent, and so will the other four choices $i\in \{1000,1001,1100,1111\}$. 

For non-archimedean norms all error terms have equal size so all choices are equivalent. Hence we may and do assume that $x_{0100}=1$. We obtain that $x_i\equiv x_{0100} \mod \pi^n$ for all $i\notin G$.
Substituting equations (\ref{eqn55}), (\ref{eqn6}) and (\ref{eqn7}) into the degree 4 equation (\ref{eqA22inP96}), we obtain
\[
12x_{0100}^4+C=0,
\text{ where }
|C|\leq  (1712\eps^4 + 2880\eps^3 + 528\eps^2 + 96\eps)|x_{0100}^4|
\]
for archimedean norms.
With $\eps\leq 0.077$ we obtain a contradiction, unless $x_{0100}=0$. In that case, we can do the same computation with $x_{1000}$ in place of $x_{0100}$, and we reach a contradiction when $\eps \leq 0.051$.

In the non-archimedean case, the degree 4 equation yields $12 \equiv 0 \mod \pi^n$ when $p\neq 2$. This is a contradiction unless $p=3$ and $|\pi^n|\leq |3|$. When $p=2$, we note that $\ep\equiv 0 \mod 2^n$ implies that $(x_{0100}+\ep)^2\equiv x_{0100}^2 \mod 2^{n+1}$. Hence, $|\pi^n|\geq |4|$ actually gives $12\equiv 0 \mod 8$, also a contradiction. 
\end{proof}
\begin{corollary}
\label{cortwoplaces}
Let $v_1,v_2$ be two places of a number field $K$, and $\psi(P) = (x_{\um})_{\um \in E} \in \P^9(K)$ without a zero coordinate. There exists at least one $\um_0\in E$ such that $x=x_{\um_0}$ satisfies
\begin{align}
\label{ineq1}
|x|_v \geq \max_{\um}|x_{\um}|\cdot \begin{cases}|2||3| \text{ if } v \text{ finite and}\\
1/27 \text { when } v \text{ infinite} \end{cases}
\end{align}
for each $v\in \{v_1,v_2\}$.
\end{corollary}
\label{explicitcor}

\begin{proof}
If (\ref{ineq1}) is violated by at most four $x$-coordinates at each of the two places, then we are free to choose any of the remaining two. Otherwise, there is one $v\in \{v_1,v_2\}$ such that (\ref{ineq1}) is violated by the $x$-coordinates in a set $T\subset \{x_{\um} \mid \um \in E\}$ of size $|T|\geq 5$. In that case, Proposition \ref{explicitprop1} tells us that $T$ is contained in the complement $T_6$ of a G\"opel quadruple. In particular, each $x$ in the G\"opel quadruple $E\setminus T_6$ satisfies (\ref{ineq1}) at $v$. Now Proposition \ref{explicitprop2} tells us that one such $x$ also satisfies (\ref{ineq1}) at the other place. We note that $1/27\leq 0.051$.
\end{proof}

\subsection{Proofs of Theorem \ref{rungethm} and Corollary \ref{cortwoplaces}}
\label{subsecthmRungeandcorollary}
After all this preparatory work, we can finally prove an upper bound on the height of our integral points considered. We prove the following slightly more specific version of Theorem \ref{rungethm} using Runge's method.

\begin{theorem}
\label{rungethm2}
Let $P = (A,\lambda,\alpha_2) \in A_2(2)(K)$ representing a triple such that the full 2-torsion $\alpha_2$ is defined over $K$, and such that the semistable reduction of $A$ is the jacobian of a smooth curve, except at most at 2 places (including necessarily the archimedean ones). We then have 
\[
h(\psi(P)) \leq 8.6 \text{ and } h_\Fcal(A) \leq 985,
\]
where $h_\Fcal$ is the stable Faltings height of $A$.
\end{theorem}

\begin{proof}
Let $S$ be the set of places including $M_K^\infty$ and the finite places $v$ such that the semistable reduction of $A$ modulo $v$ is not isomorphic to a jacobian. By assumption, $|S| \leq 2$. Furthermore, $A$ is necessarily a jacobian of hyperelliptic curve (after possible extension, see Theorem \ref{thmjacorprod}). The coordinates mentioned below refer to the ten coordinates of $\psi(P)$, normalised to belong to $K$.

For the places $v \notin S$ and not dividing 2, all the coordinates have the same valuation by Proposition \ref{propalgthetafoncetreduchorsde2} $(a)$. For the places $v \notin S$ above 2, the smallest possible ratio $|x_i|_v/|x_j|_v$ of coordinates is $|2|_v^6$ by Proposition \ref{propalgthetachar2}.

For the (at most two) places of $S$, one can choose by Corollary \ref{cortwoplaces} an index $i \in E$ such that $|x_i|_v \geq C_v \max_{j \in E} |x_j|_v$ with $C_v= |2|_v,|3|_v$ or $1/27$. We keep this choice of $i$. We thus have
\begin{eqnarray*}
h(\psi(P))  & = &  \frac{1}{[K:\Q]} \sum_{v \in M_K} n_v \log \left( \max_{j \in E} \left|\frac{x_j}{x_i} \right|_v\right) \\ & \leq & \log(27) + 6 \log(2) + \log(3) \leq 8.6. 
\end{eqnarray*}
The bound $\log(27)$ comes from the contribution of archimedean places, while $6 \log (2)$ comes from the places above 2 (if $S$ contains a place above 2, the bound obtained is smaller) and $\log(3)$ appears if a place above 3 belongs to $S$. The other places do not contribute.
We deduce the bound on the Faltings height by \cite[Corollary 1.3]{Pazuki12b}, taking into account that $g=r=2$ here and with his notations, $h_{\Theta}(A,L) = \frac{h(\psi(P))}{4}$.
\end{proof}

\begin{remark}
\label{remRungeSunits}
This proof was conceptualised in the context of Runge's method for varieties, but an alternative approach was possible to prove a version of this theorem. Indeed, with the same notations and hypotheses for $A$ and $S$, one can exhibit a curve $C$ in Rosenhain normal form such that $\Jac(C) \cong A$  (see the proof of Theorem \ref{thmjacorprod}). Its parameters $\lambda_1, \lambda_2$ and $\lambda_3$, belonging to $K$, will have $v$-adic valuation 1 for all $v \notin S$ and not dividing 2, because their squares are cross-ratios of theta constants, using Proposition \ref{propalgthetafoncetreduchorsde2}. In fact, the six Weierstrass points $\{0,1,\infty,\lambda_1,\lambda_2,\lambda_3\}$ will also be distinct modulo $v$ for those $v$ (e.g. because their differences can also be written as Rosenhain parameters and cross-ratios of fourth powers of theta constants), so in particular the $\lambda_i$ they are $v$-integral in $\P^1 \backslash \{0,1,\infty \}$ for all $v$ outside $S$ (in other words, solutions of the unit equation). These coefficients thus satisfy the hypotheses of Runge's method on $\P^1 \backslash \{0,1,\infty\}$ and one can bound their height. They then allow to  determine back the coordinates of $\psi(P)$ via \cite[Lemma 2.5]{ClingherMalmandier20}, which finally bounds the height of $\psi(P)$. 

This approach would work, but we would need to deal with similar complications (such as what happens modulo 2, where no standard Rosenhain form exists), which ultimately boils down to using theta constants again, and it is not clear it would give better bounds than the one we found, so we decided to present the results via the graph of intersection.
\end{remark}

\noindent{\bf Proof of Corollary \ref{rungecor}.}  Suppose that $P=(x_1:\ldots:x_{10})\in A_2(2)(\Q)$ corresponds to the jacobian of a hyperelliptic curve with potentially bad reduction at a single prime $p>2$. We will apply Theorem \ref{rungethm2} with $K=\Q$ and $S=\{\infty,p\}$. First, by Proposition \ref{propalgthetafoncetreduchorsde2}, we can (and do) scale the $x_i$ such that $x_1,\ldots,x_{10}\in \Z$ and at most one of them satisfies $v_p(x_i)>0$. By transitivity of the $\mathrm{Sp}_4(\F_2)$-action, we may suppose this is $x_5$. Similarly, after scaling each $x_i$ by a power of $\ell$, we have $v_\ell(x_i)=0$ for all primes $\ell \notin \{2,p\}$. We also scale by a power of 2 such that one $x_i$ satisfies $v_2(x_i)=0$. By Proposition \ref{propalgthetachar2} (b), at least four $x_i$ now satisfy $v_2(x_i)=0$, and we may suppose by 2-transitivity of the $\mathrm{Sp}_4(\F_2)$-action that one of them is $x_6$. Now $x_6=\pm1$. The height bound $\mathrm{log}(2^6\cdot 3 \cdot 27)$ from  Theorem \ref{rungethm2} now implies that $x_5$ is divisible by a power of $p$ of size at most $2^6\cdot 3 \cdot 27$. Finally, by Proposition \ref{propalgthetachar2} (a), we find (after possibly scaling further by a minus sign) that
 \begin{itemize}
     \item[(i)] $x_6=1$,
     \item[(ii)] $x_5=q\cdot x_5'$, where $1\leq q\leq 2^6\cdot 3 \cdot 27$ is a prime power and 
     \item[(iii)] $x_1,x_2,x_3,x_4,x_5',x_7,x_8,x_9,x_{10}\in \{\pm 1,\pm 2,\pm 2^2,\ldots,\pm 2^6\}$.
 \end{itemize}
 Now a priori there appear to be too many possibilities to check by computer, but recall that the $x_i$ satisfy a bunch of linear relations. In fact, following Igusa \cite[p 396, 397]{Igusa64}, we define the map
 \[
 \phi\colon\P^9\longrightarrow \P^4,\quad (x_1:\ldots:x_{10})\mapsto (x_6:x_5:x_1:-x_6-x_7:-x_6-x_9),
 \]
 mapping $A_2(2)$ isomorphically onto the threefold $Y\subset \P^4$ defined by
 \begin{align*}
 y_1^2y_2^2& - 2y_1^2y_2y_3 - 2y_1y_2^2y_3 + y_1^2y_3^2 - 2y_1y_2y_3^2 + y_2^2y_3^2 - 4y_1y_2y_3y_4 \\&- 4y_1y_2y_3y_5 - 2y_1y_2y_4y_5 - 2y_1y_3y_4y_5 - 
    2y_2y_3y_4y_5 + y_4^2y_5^2=0.
 \end{align*}
 We now search for solutions $(y_1:\ldots:y_5)\in Y(\Q)$ satisfying $y_1=1$, $y_2=q\cdot y_2'$ where $1\leq q \leq 2^6\cdot 3\cdot 27$ is a prime power and $y_2',y_3,y_4+1,y_5+1\in \{\pm1,\pm2,\ldots,\pm2^6\}$. This amounts to evaluating the quartic polynomial defining $Y$ at a total of 27.736.352 values $(y_1,\ldots,y_5)$, which is sufficiently small to do on a computer in a matter of minutes. The inverse of $\phi$ is given by $\psi\colon (y_1:\ldots:y_5)\mapsto (y_3:y_3+y_5:y_1+y_2+y_3+y_4+y_5:y_3+y_4,y_2:y_1:-y_1-y_4:-y_2-y_4:-y_1-y_5:-y_2-y_5)$. We apply $\psi$ to each solution, remove those $(x_1:\ldots:x_{10})$ with a zero coordinate (they correspond to boundary points) and check whether (iii) is satisfied. This leaves only two options, which are $\mathrm{Sp}_4(\F_2)$-equivalent. We thus find only one possible hyperelliptic curve, corresponding to the point
 \[
 P=( -4: 1: -4: 1: -9: -1: -4: 4: -4: 4 ).
 \]
 By Thomae's formulae used in the proof of Theorem \ref{thmjacorprod}, we find that $P$ corresponds to a hyperelliptic curve $C\colon y^2=x(x-1)(x-\lambda_1)(x-\lambda_2)(x-\lambda_3)$ satisfying $\lambda_1^2=16$, $\lambda_2^2=4$ and $\lambda_3^2=4$. This yields 2 non-isomorphic hyperelliptic curves. Using the \texttt{genus2reduction} function in \texttt{Sage} for genus 2 hyperelliptic curves, we find that one of the two has potentially bad reduction at 5, so it cannot correspond to $P$. The other curve,
 \[
 C\colon y^2=x(x-1)(x^2-4)(x-4),
 \]
 must therefore correspond to the point $P$. Indeed, as predicted by Proposition \ref{propalgthetafoncetreduchorsde2}, $C$ has potentially good reduction at all primes $p>3$ and does not have potentially good reduction at 3. However, we find that $C$ also does not have potentially good reduction at 2, leaving us with no solutions.
 
 Finally, we need to consider $p=2$. Now  we need to drop the assumption that the valuation at 2 of the coordinates is at most 6. In return, we obtain from the proof of the above theorem a smaller height bound: $\mathrm{Log}(2\cdot 3 \cdot 27)$. Since $2^8>2\cdot 3 \cdot 27$, we may now assume that $(x_1:\ldots:x_{10})\in A_2(2)(\Q)$ satisfies\begin{itemize}
     \item[(i)] $x_6=1$ 
     \item[(ii)] $x_1,x_2,x_3,x_4,x_5,x_7,x_8,x_9,x_{10}\in \{\pm 1,\pm 2,\ldots,\pm 2^7\}$.
 \end{itemize}
 This is an even faster computation, and we find no solutions. $\qed$

\section{Baker's method with blowing-down cycles}
\label{secBaker}

In this section, we will use Baker's method to prove Theorem \ref{bakerthm}.

To this end, we will use a peculiar property of $A_2(2)^S$: for every component of $D_i \cap D_j$, there exists a rational function on $A_2(2)^S$ of the shape $x_k/x_\ell$ for some other indices $k,\ell \in E$ blowing $D_i\cap D_j$ down to $1$ or $-1$.

The classical Baker's method for curves (see \cite{Bilu95} for an overview) relies on the existence of enough rational functions supported on the divisor $D$ with respect to which integrality is defined. In this higher-dimensional situation, we will show that this property means that we will be able to apply Levin's generalisation of Baker's method (\cite{Levin14}, also recalled in paragraph \ref{subsecLevin} below) as soon as a point is $v$-close to two divisors for some place $v$, except when the point belongs to the inverse image of 1. In a more common context, this would be impossible and we would need to assume that the point is $v$-close to 5 divisors (to be close to a specific point in a finite family), following Levin's method.

Here, we will keep in mind the graph of intersection, and follow the formalism of \cite[Section 5, Baker's method]{LeFourn5}.

First, let us recall that the linear equations defining $A_2(2)^S$ all involve azygous quadruples. There are 15 azygous quadruples, and using the transitivity of the action of $\Sp_4(\F_2)$, one can thus exhibit 15 total linear equations involving four coordinates (and the signs are changed accordingly). The key idea for starting the argument is summed up in the proposition below.\\

\noindent \textbf{Notation.}
For convenience here, we index elements of $E$ from 1 to 10 in increasing order of the fourtuples seen as binary expansions (e.g. $(0000)$ is 1 and $(0110)$ is 6), and define the indices of coordinates and dot action accordingly.

\begin{proposition}
\label{propstartBaker}
For any distinct pair $\{i,j\} \in \{1,\cdots, 10\}$ and any point $x \in A_2(2)^S(\C)$, if $x_i = x_j=0$, then there are exactly two disjoints pairs $\{k_1,\ell_1\}$ and $\{k_2,\ell_2)$ completing $\{i,j\}$ into an azygous quadruple, and then 
\[
x_{k_1} = \varepsilon_1 x_{\ell_1}, \quad x_{k_2} = \varepsilon_2 x_{\ell_2},
\]
with $\varepsilon_1, \varepsilon_2 \in \{\pm 1 \}$ determined unambiguously by $(i,j,k_1,\ell_1)$ and $(i,j,k_2,\ell_2)$.

In other words, for $\phi_1 := \varepsilon_1 \frac{x_{k_1}}{x_{\ell_1}}$,  we have $\phi_1((D_i \cap D_j) \backslash \operatorname{supp}\phi_1) = \{1\}$ and the same happens for $\phi_2$.

Geometrically, the intersection $D_i \cap D_j$ has two irreducible components, each of dimension 1: the first is $Z_{\{i,j,k_1,\ell_1\}}$ (blown down to 1 by $\phi_2$) and the second is $Z_{\{i,j,k_2,\ell_2\}}$ (blown down to 1 by $\phi_1$). It is not possible to use these functions in the reverse order: indeed, $Z_{\{ i,j,k_i,\ell_i\}} \subset \operatorname{supp} \phi_i$ for $i=1,2$.
\end{proposition}

\begin{proof}
 The $\odot$-action of  $\mathrm{Sp}_4(\F_2)$ is 2-transitive, so the pairs $(k_1,\ell_1),(k_2,\ell_2)$ are determined for each intersection $D_i\cap D_j$ by what they are for $D_1\cap D_2$. Similarly, the signs are determined by the corresponding signs on $D_1\cap D_2$ via (\ref{eqdotactioncoordinates}) (we will make this precise later). It is thus enough to prove the result for $i=1$ and $j=2$. The formal manipulation of the equations $x_1=0$ and $x_2 = 0$ over $\Z$, then gives that $x_6=-x_9$ and $2x_5=-2x_{10}$, so $x_5 = -x_{10}$.
 
 The claims of the previous paragraph are readily checked by computations; see also the graph of intersection. 
\end{proof}

\begin{remark}
One needs to be careful with the notations: the rational function we will use for $Z_{\{i,j,k_1,\ell_1\}}$ is not $\phi_1$ but $\phi_2$, and in fact  $Z_{\{i,j,k_1,\ell_1\}} \cap \operatorname{supp}{\phi_2} = \{Q\} = Z_{\{i,j,k_2,\ell_2\}} \cap \operatorname{supp}{\phi_2}$, where  $Q = (0:0:1:1:0:0:-1:-1:0:0)$.
\end{remark}

This unusual phenomenon (a one-dimensional intersection sent to a point by a rational function with support in the union of the divisors) will allow us to push much further the ordinary application of higher-dimensional Baker's method.


\subsection{Adaptation of Levin's generalisation of Baker's method}
\label{subsecLevin}

We give here first a quick overview of the higher-dimensional Baker's method due to Levin \cite{Levin14}, before explaining how it can be improved upon here. The notations employed are reminiscent of (but do not refer to the exact same objects as) the other sections, because we give an explanation in a general case.

Let us assume we have $X$ a normal projective variety over $K$ (with an implicit model over $\Ocal_K$ to define integral points properly), $D = \bigcup_{i=1}^n D_i$ a union of ample effective divisors and $P \in (X \backslash D)(\Ocal_{K,S})$. One wishes to bound the height $h_D(P)$ of $P$ relative to $D$ (all global and local heights are assumed precisely defined and fixed below). To shorten the explanation, the symbol $a \geq b$ (resp. $a \gg b$) will refer to the existence of computable absolute constants $C,C'>0$ independent of $a,b$ such that $a \geq b - C'$ (resp. $a \geq C b - C'$). 

For each divisor $D_i$, there exists a place $v \in S$ such that 
\[
(\ast)_{v,i}: \quad h_{D_i,v}(P) \geq \frac{1}{|S|} h_{D_i}(P),
\]
because the local heights are 0 at places outside $S$ and the sum of local heights gives the global one. Assume there is a place $v$ of $S$ such that $(\ast)_{v,i}$ holds on all $i \in I$ and the intersection $Z_I = \bigcap_{i \in I} D_i$ is finite. Assume furthermore that for each of the points $Q$ of $Z_I$, there is a non-constant rational function $\phi$ with support in $D$ sending it to $1$. After a quick manipulation of local heights, we thus obtain, for a good choice of $Q$ and $\phi$, if $P \notin \operatorname{supp} \phi$,
\begin{equation}
\label{eqseriesineqBaker}
h_{1,v}(\phi(P)) \geq h_{Z_I,v}(P) \gg h_{D_i}(P) \gg h_D(P) \gg h(\phi(P)). 
\end{equation}
The left inequality is due to the fact that if $P$ is $v$-close to $Q$, $\phi(P)$ is (even more) $v$-close to $1 =\phi(Q)$. The second one is using the $(\ast)_{v,i}$, and the two last ones come from the ampleness of the divisors considered. 

Now, if $\phi(P) \neq 1$, this means we have a point of $\P^1(K)$ $v$-close to 1, but $\phi(P) \in \Ocal_{K,S}^*$ (up to a finite number of possible constants) because $P \in (X\backslash D)(\Ocal_{K,S})$  and $\phi$ is supported on $D$. The theory of linear forms in logarithms thus gives bounds of the shape 
\begin{equation}
\label{eqlinforms}
h_{1,v}(\phi(P)) \ll C_1(K,S) \log\max (h(\phi(P)),1)
\end{equation}
with $C_1$ effective in $K$ and $S$.
Combining with \eqref{eqseriesineqBaker}, we then obtain a bound of the shape $h(\phi(P)) \ll C_2(K,S)$, and finally $h_D(P) \ll C_3(K,S)$ by reusing \eqref{eqlinforms} and the left inequality of \eqref{eqseriesineqBaker}.

The hypothesis of existence of good functions $\phi$ supported on $D$ (they are often called $D$-units) is geometric, but the existence of a good $v$ and $I$ has to rely on combinatorial arguments. As Levin found, the pigeonhole principle gives the sufficient condition $(m_\Bcal-1) |S|<n$, where $m_\Bcal$ is the minimum number for which any intersection of $m_\Bcal$ divisors $D_i$ is finite. Indeed, in this case, either there is a place $v$ satisfying the hypotheses above, or there is a divisor $D_i$ which is $v$-far from $P$ for every $v \in S$, in which case one can fall back to the conclusion of Runge's method. In our case ($n=10$, $m_\Bcal=5$), a straight application would not give any improvement to our refined Runge's method (and with far worse bounds due to the theory of linear forms in logarithms). 

The basis of our improved application here is the following Lemma, inspired by Lemma 10 of \cite{Levin14}.

\begin{lemma}
\label{lemLevinmod}
Let $C$ be a reduced prime cycle on a normal projective variety $X$ over a number field $K$, and $\phi \in K(X)$ such that $\phi(C \backslash C_\phi) = 1$ where $C_\phi := C \cap \operatorname{supp}\phi$, in other words $C \backslash C_\phi$ is blown down to 1 by $\phi$.

Then, for every $P \in (X \backslash \operatorname{supp} \phi)(K)$ such that $\phi(P) \neq 1$ and every $v \in M_K$,
\[
h_{C,v}(P) \leq h_{1,v}(\phi(P)) + h_{C_\phi,v}(P) +  O_v(1),
\]
where the sum of all errors $O_{v}(1)$ over the $v \in M_K$ can be bounded by an effectively computable constant, independent of $P$.
\end{lemma}

\begin{proof}
Let $\pi : \widetilde{X} \rightarrow X$ be a blowup of $X$ along the support of $\phi$, such that $\phi$ extends to a morphism of projective varieties $\widetilde{\phi} : \widetilde{X} \rightarrow \P^1$ and $\widetilde{\phi}  = \phi \circ \pi$ as rational functions.
The point $P$ lifts to a unique point $\widetilde{P} \in \widetilde{X}(K)$ because it does not belong to the support of $\phi$, and then by functoriality of local heights,
\[
h_{C,v}(P) = h_{C,v}(\pi(\widetilde{P})) = h_{\pi^*C, v}(\widetilde{P}) + O_v(1)
\]
and similarly for $C_\phi$, and $h_{1,v}(\phi(P)) = h_{\widetilde{\phi}^*1,v}(\widetilde{P}) + O_v(1)$. Now, by construction, $\phi$ blows down $C \backslash C_\phi$ to 1 so as ideal sheaves, $\pi^* C \subset \widetilde{\phi}^* 1 + \pi^* (C_\phi)$. Indeed, this inclusion is clear outside of $\pi^* \operatorname{supp} \phi$ by our hypothesis on $\phi$, and it holds on $\pi^* \operatorname{supp} \phi$ by definition of $C_\phi$. Combining with the previous inequalities, we obtain the result.
\end{proof}

In our situation, such a function $\phi$ exists for any of the two irreducible components of $D_i \cap D_j$, by Proposition \ref{propstartBaker}. For our application of Baker's method, it means that as soon as the set $I$ of indices is of order at least 2, we can apply a series of inequalities similar to \eqref{eqseriesineqBaker}, which makes $|S| < 10$ (instead of $4|S|<10$) the sufficient condition for our modified method to apply. As in Levin's method, there will be an exclusion set: the set of $P$'s for which $\phi(P) = 1$, which cannot be dealt with in this way. After applying the method, we will prove that those exclusion sets parametrise very specific abelian surfaces.

Furthermore, the term $h_{C_\phi,v}(P)$ compels us to deal with the cases where $P$ is $v$-close to one of the fifteen points obtained as $Z_I$ for $I$ the complement of a Göpel quadruple: if this height is large we have many possible choices for a $D$-unit. If it is not, we can (up to controlled error) act as if $h_{C,v} (P) \leq h_{1,v} (\phi(P))$ and execute Baker's method as announced.



\subsection{Explicit Baker's method outside exclusion sets}
In this subsection we prove Theorem \ref{bakerthm} by first bounding $h(\psi(P))$.




\begin{lemma}
\label{lemlinforms}
Let $K$ be a number field of degree $d$ and $S$ a set of places of $K$ of size $s$ containing $M_K^\infty$. Then, for any $x \in \Ocal_{K,S}^*$ and any place $v$ of $K$, for a fixed $\alpha \in K^*$ such that $\alpha x \neq 1$, one has
\begin{equation}
\label{eqlinform}
- \log \left| \alpha x - 1\right|_v \leq C_1(d,s) R_S N_v \max(h(\alpha),1) \log( C_2(d,s) h(x)) 
\end{equation}
with $C_1(d,s)$ and $C_2(d,s)$ effectively computable, $R_S$ is the regulator of $\Ocal_{K,S}^*$, $N_v$ is the norm of the prime ideal corresponding to $v$ if $v$ is finite and 1 otherwise.
\end{lemma}

\begin{remark}
\label{remarkboundsexplicit}
Under the constraints with which we are working in this paper ($d \leq 18$ and $s \leq 9$), using the effective values below, we have $C_1(d,s) \leq 8 \cdot 10^{35}$ and $C_2(d,s) \leq 5 \cdot 10^{13}$ in the archimedean case, and $C_1(d,s) \leq 7 \cdot 10^{59}$ and $C_2(d,s) \leq 3 \cdot 10^{12}$ in the non-archimedean case.
\end{remark}

\begin{proof}
By \cite[Lemma 1]{BugeaudGyory96}, we can find a basis $(\varepsilon_1, \cdots, \varepsilon_{s-1})$ of $\Ocal_{K,S}^*$ up to torsion such that 
\[
\prod_{i=1}^{s-1} h(\varepsilon_i) \leq \frac{(s-1)!)^2}{2^{s-2} d^{s-1}} R_S.
\]
For any $x \in \Ocal_{K,S}^*$, we can write
\[
x = \zeta^{b_0} \prod_{i=1}^{s-1}  \varepsilon_i^{b_i},
\]
with $\zeta$ a root of unity in $K$ and integers $b_1, \ldots, b_{s-1}$. By the same Lemma and choice of basis (working out the values of $c_4,c_6$ and $\delta_K$ from inside the paper), we then have
\[
B := \max_{1 \leq i \leq s-1} |b_i| \leq 53 \frac{((s-1)!)^2}{2^{s-3}} d^2 \log(6d) h(x).
\]
By \cite[Theorem 2.2]{BugeaudLinearForms}, in the archimedean case, one obtains \eqref{eqlinform} with 
\begin{eqnarray*}
C_1(d,s) & = & 12 \pi \times 30^{s+4}(s+1)^{5.5} d^{2} \log(ed) \frac{((s-1)!)^2}{2^s} \text{ and }\\
C_2(d,s) & = & 53 e s \frac{((s-1)!)^2}{2^{s-3}} d^2 \log(6d).
\end{eqnarray*}

In the non-archimedean case, by \cite[Theorem 2.10]{BugeaudLinearForms} one obtains \eqref{eqlinform} with 
\begin{eqnarray*}
C_1(d,s) & = & 12 (6(s+1)d)^{2s+2} \log(e^5 s d) \frac{((s-1)!)^2}{2^{s-2} d^{s-1}}  \\
C_2(d,s) & = & 53 \frac{((s-1)!)^2}{2^{s-3}} d^2 \log(6d).
\end{eqnarray*}
(the term $p^d - 1$ in the estimates can be replaced by $N_v-1$, see \cite[p. 174]{BugeaudLinearForms}).

In the archimedean case, one has to use the inequality of \cite[p. 77]{BugeaudLinearForms}, and backtracking the values of $c_8,c_{10},c_{11},B$ there (taking into account that the heights are not logarithmic in that reference, and $\log H \geq h(\alpha)$), we obtain the inequality above with (simplified) constant
\[
C_1(d,s) = 240000 \times d \log(d)^s ((s-1)!)^2 2000^s (s+1)^{3s+9}
\]
and 
\[
C_2(d,s) = \frac{8 d ((s-1)!)^2}{2^s}.
\]
\end{proof}

We now define the local heights involved in the computations. For any cycle of the shape $Z_I$, and any $P \in A_2(2)^S(K)$ not in $Z_I$ we define (as is natural)
\[
h_{Z_I,v}(P) = - \log \left( \frac{\max_{i \in I} |x(P)_i|_v}{\|x(P)\|_v} \right)
\]
where $x(P)=(x(P)_1, \cdots, x(P)_{10}) \in K^{10}$ is any choice of projective coordinates of $\psi(P)$, and $\|x(P)\|_v = \max_{i \in E} |x_i|_v$. In $\P^1$, we simply have to define for $x \in K, x \neq 1$:
\[
h_{1,v}(x) = \max(0,- \log(|x - 1|_v)).
\]
Afterwards, one defines as usual the global heights via 
\[
h_{Z_I}(P) = \sum_{v \in M_K} \frac{n_v}{[K:\Q]} h_{Z_I,v}(P).
\]

Notice that all the divisors $D_i$ are linearly equivalent, and more precisely that  $h_{D_i}(P) = h (\psi(P))$ for any $P$ not in $D_i$ by manipulating the global height formula. We will denote this common height by $h(P)$ for simplicity later on.\\

\noindent \textbf{Proof of Theorem \ref{bakerthm}.} We consider a point $P=(A,\lambda,\al_2)\in A_2(2)(K)$  representing an abelian surface (with full 2-torsion defined over $K$) whose semistable reduction at all places outside $S$ is a jacobian of hyperelliptic curve (over a possible finite extension). We can assume $h(P) > 1000$ for convenience as the final bounds obtained are much larger. We also assume throughout that $s=|S|<10$.

The local height $h_v(P)$ at all places not in $S$ is very small (see our analysis of the difference between integral points in the sense of the projective scheme $(A_2(2)^S \backslash D)$ and integral points in terms of moduli in Propositions \ref{propalgthetafoncetreduchorsde2} and Propositions \ref{propalgthetachar2}) and $s$ is at most 9, so we can assume that the contributions to the global height of all places not in $S$ is at most $h(P)/10$. 

By the pigeonhole principle, there are two distinct indices $i,j \in E$ and a place $v \in M_K$ such that 
\[
h_{D_i,v}(P) \geq \frac{1}{10} h(P), \quad h_{D_j,v}(P) \geq \frac{1}{10} h(P).
\]

By 2-transitivity of the action of $\Sp_4(\F_2)$ (which preserves the global height), one can assume that $i=1$ and $j=2$. We thus have, by definition, for the (reducible) cycle $C = D_1 \cap D_2$, 
\[
h_{C,v}(P) \geq \frac{1}{10} h(P).
\]
According to the graph of intersection, this cycle can be written as $C = C_1 \cup C_2$, where $C_1$ is given by the equations $x_1=x_2= x_5 = x_{10} = 0$ and $C_2$ by $x_1 = x_2 = x_6 = x_9 = 0$. One of them, let us say $C_k$ ($k\in \{1,2\}$) thus satisfies
\[
h_{C_k,v}(P) \geq \frac{1}{20} h(P).
\]
If that is the case for both, we obtain 
\[\
h_{Q,v}(P) = \min(h_{C_1,v}(P),h_{C_2,v}(P)) \geq \frac{1}{20} h(P) >0,
\]
where $Q=(0:0:1:1:0:0:-1:-1:0:0)$ is the unique point of intersection of $C_1$ and $C_2$, associated to the complement of Göpel quadruple $(1,2,5,6,9,10)$. We will deal with the particular case where $h_{Q,v}(P)$ is large (more precisely $h_{Q,v}(P) \geq h(P)/40$) later, so we assume for now that $h_{Q,v}(P) < h(P)/40$.

For $k=1$, let us fix $\phi_1 = - \frac{x_6}{x_9}$ and for $k=2$, we fix $\phi_2 = - \frac{x_5}{x_{10}}$. Both these functions satisfy Lemma \ref{lemLevinmod} and the intersection of the cycle with the support of the corresponding function is the point $Q$ as above in both cases. Applying the Lemma tells us of the existence of inequalities up to constants, but of course we need to make everything explicit. One of the equations defining $A_2(2)^S$ over $\Q$ is $x_1 - x_2 - x_6 - x_9 = 0$, and another is $x_1 - x_2 - x_5 - x_{10} = 0$. 

From now on, $(x_1, \cdots, x_{10}) \in K^{10}$ denotes a choice of projective coordinates of $P$. If $h_{C_1,v}(P) \geq \frac{1}{20} h(P) > 2 h_{C_2,v}(P)$, we thus have $|x_6|_v$ or $|x_9|_v$ strictly larger than $\max(|x_1|_v,|x_2|_v,|x_5|,|x_{10}|_v)$. In the non-archimedean case, using the first equation,  $|x_6|_v = |x_9|_v$ so for each $i \in \{1,2,5,10\}$, $ - \log (|x_i|_v/|x_6|_v) \geq \frac{1}{40}h(P)$. Consequently, $- \log |\phi_1(P) - 1|_v \geq \frac{1}{40}h(P)$.

The same thing holds in the archimedean case up to an error $\log(2)$.

Furthermore, $\phi_1(P) = - \frac{x_6}{x_9}$ is a unit in $\Ocal_{K, S'}^*$ where $S' = S \cup S_2$  with $S_2$ the set of places of $K$ above 2, by Proposition \ref{propalgthetafoncetreduchorsde2}. More precisely, raising it to the power $h_K$, we can thus write it as an $\Ocal_{K,S}^*$-unit times an element with non-trivial valuation only at primes above 2, and of height at most $6 h_K \log(2)$ by Proposition \ref{propalgthetachar2}. 

We have to assume from now on that $\phi_1(P) \neq 1$ (the study of this case being postponed to the next section). For archimedean places, we have
\[
- \log |(x_6/x_9)^{h_K} - 1|_v \geq - \log |(x_6/x_9) - 1|_v - h_K \log (2),
\]
(the second term disappears for finite places).

Applying Lemma \ref{lemlinforms}, we thus obtain 
\[
 - \log \left| -\left(\frac{x_6}{x_9}\right)^{h_K} - 1 \right|_v \leq C_1(d,s) 6 h_K \log(2) R_{S} N_v \log( C_2(d,s) h_K h(x_6/x_9)),
\]
which leads to 
\[
h(P) \leq 200 (C_1(d,s) R_{S} h_K N_v \log (C_2(d,s) h_K h(P)))
\]
whether $v$ is finite or not (the term $h_K \log(2)$ in the archimedean case being absorbed in the cruder bound here).

A coarse but straightforward manipulation of this inequality leads to 

\begin{equation}
\label{eqprovbound}
h(P) \leq 400 C_1(d,s) R_{S} h_K N_v \log^*(R_{S} h_K N_v)  \log (200 C_1(d,s)C_2(d,s)).
\end{equation}
By Remark \ref{remarkboundsexplicit}, we thus obtain the explicit bound 
\[
h(P) \leq 10^{66} R_S h_K N_v \log^*(R_{S} h_K N_v).
\]
The Faltings height is deduced from it by \cite[Corollary 1.3]{Pazuki12b} again (using crude bounds and keeping the factor $10^{66}$ here).

The same estimate holds in the case $h_{C_2,v}(P) \geq h(P)/20$ (and $h_{C_1,v}(P) < h(P)/40$). 

It remains to study the case where $P$ is $v$-close to the point $Q$, in particular at none of the 6 indices $1,2,5,6,9,10$ the norm $\|x\|_v$ is attained.

By Proposition \ref{explicitprop1}, the 4 other indices thus have norm $\|x\|_v$ (finite case) or at least $1/27 \|x\|_v$ (archimedean case). Furthermore, the point $Q$ is sent to 1 via the function $x_7/x_8$ (we actually have many possibilities here) and we can apply the same method here using the equation $x_7 - x_9 + x_{10} - x_8=0$ to realise explicitly Lemma \ref{lemLevinmod}. We infer the exact same type of bounds, apart from one difference: we assume only $h_{Q,v}(P) \geq h(P)/40$ and not $h(P)/20$, so at an intermediary step we get a doubled right-hand side. However, this factor of 2 is absorbed in the very crude approximation we do afterwards, which still gives $10^{66}$ as the next power of 10.  Finally, the worst case of this combined estimate is when the place $v$ comes from the prime ideal with largest norm of $S$, which gives the final estimate.

We have now proved the theorem for points $P$ such that there is no $M\in \mathrm{Sp}_4(\F_2)$ with $\frac{x_1}{x_2}(M\cdot P)=1$ or $\frac{x_7}{x_8}(M\cdot P)=1$. In the next section we will investigate this exceptional set further in order to complete the proof of Theorem \ref{bakerthm}.

\subsection{Determining the exclusion sets}

In this section we find a moduli interpretation for the exclusion sets that appeared in the application of Baker's method. To this end, we first need to make precise the signs in Proposition \ref{propstartBaker}. Recall that on every intersection $D_i\cap D_j$ there exist two pairs of indices $(k_1,\ell_1)$ and $(k_2,\ell_2)$ such that 
\[
x_{k_1}=\pm x_{\ell_1} \text { and } x_{k_2} = \pm x_{\ell_2} \text{ on } D_i\cap D_j,
\]
where each of the two pairs $(x_{k_i},x_{\ell_i})$ vanish identically on one of the two irreducible components of $D_i\cap D_j$.
In order to determine the functions mapping irreducible components of $D_i\cap D_j$ to 1, we need to know these signs. Recall that the $\odot$-action of  $\mathrm{Sp}_4(\F_2)$ is 2-transitive, and so the pairs $(k_1,\ell_1),(k_2,\ell_2)$ are determined for each intersection $D_i\cap D_j$ by what they are for $D_1\cap D_2$. Similarly, the signs are determined by the corresponding signs on $D_1\cap D_2$ via (\ref{eqdotactioncoordinates}). We now make this precise.

We continue to denote the coordinates by $x_1,\ldots,x_{10}$, where a subscript $i$ refers to the $i$th element of $E$ in its binary order. We denote this element of $E$ by $\um_i$. 

\begin{lemma}
Consider distinct $i,j\in \{1,\ldots,10\}$ and find $M\in \mathrm{Sp}_4(\F_2)$ such that $\um_1\odot M = \um_i$ and $\um_2\odot M=\um_j$. Then the sign $\ep$ in the equation
\[
\frac{ x_i(M^{-1}\tau)}{x_j(M^{-1}\tau)}=\ep\frac{x_1(\tau)}{x_2(\tau)}
\]
is independent of the choice of $M$ and equals
\[
(-1)^{(\um_i+\um_j)_1(\um_i+\um_j)_3+(\um_i+\um_j)_2(\um_i+\um_j)_4}.
\]
\end{lemma}
\begin{proof}
Consider $i,j$ and $M$ as in the statement. Then
\[
\frac{ x_i(M^{-1}\tau)}{x_j(M^{-1}\tau)}=(-1)^{(\um_1-\um_2) \cdot {}^t((B{}^t A)_0). (C{}^t D)_0))}\frac{x_1(\tau)}{x_2(\tau)}
\]
by (\ref{eqdotactioncoordinates}). Note that $(\um_1-\um_2) \cdot {}^t((B{}^t A)_0). (C{}^t D)_0))=C_{21}D_{21}+C_{22}D_{22}$. But $\um_1\odot M=\um_i$ implies that $(({}^tCA)_0,{}^tDB)_0)=\um_i$ and $\um_2\odot M=\um_j$ means that $(C_{21},C_{22},D_{21},D_{22})-\um_i=\um_j$. 
\end{proof}
We denote this sign by $\ep(i,j)$. By definition, we have
\[
\ep(i,j)\ep(j,k)=\ep(i,k) \text{ for all } i,j,k\in \{1,\ldots,10\}.
\]
Note that $\ep(1,2)=1$ (by definition); we also compute that $\ep(5,10)=-1$, $\ep(6,9)=-1$ and $\ep(7,8)=1$. This shows that the functions $\phi$ in Proposition \ref{propstartBaker} are indeed of the form $\ep(i,j)\frac{x_i}{x_j}$. The set of points for which Baker's method does not work is thus
\[
\left\{P\in A_2(2) \mid \ep(i,j)\frac{x_i}{x_j}(P)=1 \text{ for some pair }(i,j)\right\}.
\]

We are now ready to describe this exceptional set.
%
%
%

\begin{proposition}
\label{exclsetsprop}
Consider $\tau=(A,\lambda,\al_2)$ defined over a number field $K$ containing $A[2]$, such that $A$ is the jacobian of a hyperelliptic curve. Then there exists a pair $(i,j)$ such that
\[
x_i(\tau)=\ep(i,j)x_j(\tau)
\]
if and only if there is a degree 2 extension $L/K$ such that $(A,\lambda)$ is isogenous to a product of elliptic curves by an isogeny defined over $L$ with kernel $(\Z/2\Z)^2$.
\end{proposition}
\begin{remark}
The sign here really is essential. This is not true when $\ep(i,j)$ is replaced with $-\ep(i,j)$.
\end{remark}

It suffices to prove this theorem for $(i,j)=(1,2)$. The following proposition provides the first step.
\begin{proposition}
Consider a pair $(i,j)$ and choose $N_{i,j}\in \mathrm{Sp}_4(\F_2)$ such that $ \{\um_i,\um_j\}\odot N_{ij}=\{\um_1,\um_2\}$. We have
\[
x_i=\ep(i,j)x_j
\]
if and only if $P=(x_1:\ldots:x_{10})$ satisfies $N_{i,j}^{-1}MN_{i,j}\cdot P=P$, where  
\[
M:=\begin{pmatrix}1& 0& 0 &1 \\ 1& 1& 1& 0\\0& 0& 1& 1\\0& 0 &0&1\end{pmatrix}\in \mathrm{Sp}_4(\F_2).
\]
\end{proposition}
\begin{proof}Again, it suffices to consider the case $(i,j)=(1,2)$. 

So suppose that $x_1=x_2$. Then the equations between the $x_i$ immediately tell us that $x_6=-x_9$ and $x_5=-x_{10}$. One can compute that $M$ is the unique matrix in $\mathrm{Sp}_4(\F_2)$
that interchanges the pairs $(\um_1,\um_2),(\um_6,\um_9),(\um_5,\um_{10})$ and leaves the remaining four $\um_i$ fixed. (And there is no non-identity matrix also leaving at least one of the pairs fixed.) We write $\phi(i,M):=(-1)^{(\um_i) \cdot {}^t((B{}^t A)_0). (C{}^t D)_0))}$. Now the action of $M$ on the projective point $P:=(x_1:\ldots:x_{10})$ is given by
 \begin{align*}
 M\cdot P=&(\phi(1,M)x_2:\phi(2,M)x_1:\phi(3,M)x_3:\phi(4,M)x_4:\phi(5,M)x_{10}:\\&\phi(6,M)x_9:\phi(7,M)x_7:\phi(8,M)x_8:\phi(9,M)x_6:\phi(10,M)x_5).
 \end{align*}
 Given our explicit $M$, we compute that in fact $\phi(i,M)=1$ for $i\in \{1,2,3,4,7,8\}$ and $\phi(i,M)=-1$ for $i\in \{5,6,9,10\}$, so that the equalities $x_1=x_2$, $x_6=-x_9$ and $x_5=-x_{10}$ yield $
 M\cdot P=P$, as desired. 
 
 Conversely, when $M$ fixes $P$ we have $x_1=x_2$ by definition.
 \end{proof}
 We note here that for each $N\in \mathrm{Sp}_4(\F_2)$, if $\um \odot N=\um$ then also $\um\odot N^{-1}=\um$, which implies that $({}^tAB)_0=({}^tCD)_0=0$ and hence $\phi(\um,N)=1$, so the equalities $\phi(i,M)=1$ for $i\in \{3,4,9,10\}$ were not surprising. \\

 \begin{lemma}
 Suppose that $C/K$ is a genus 2 hyperelliptic curve such that all Weierstrass points of $C$ are defined over $K$. Then $C$ admits a $K$-automorphism of order 2 unequal to the hyperelliptic involution if and only if $J(C)$ is $(\Z/2\Z)^2$-isogenous over $K$ to a product of elliptic curves over $K$.
 \end{lemma}
 \begin{proof}
 By \cite[Lemma 2]{shaska}, $C$ admits such an automorphism if and only if $C$ has a model of the form
 \[
 C:\; Y^2=(X^2-a^2)(X^2-b^2)(X^2-c^2),
 \]
 in which case the automorphism is given by $\psi: (X,Y)\mapsto (-X, Y)$. If this is the case, we obtain quotient maps $\pi_1:C\to E_1$ and $\pi_2: C\to E_2$, where $\pi_1$ is the quotient by $\psi$ and $\pi_2$ is the quotient by $\psi\circ \iota$ ($\iota$ here denotes the hyperelliptic involution). Here $E_1$ and $E_2$ have explicit models given by
 \begin{align*}
E_1:&\; Y^2 = (X-a^2)(X-b^2)(X-c^2) \text{ and }\\
E_2:&\; Y^2 = (1-a^2Z)(1-b^2Z)(1-c^2Z).
\end{align*}
Now $\pi_1$ and $\pi_2$ induce an isogeny $J(C)\to E_1\times E_2$ with kernel $(\Z/2\Z)^2$.

Conversely, suppose that we start with an isogeny $J(C)\to E_1\times E_2$ with kernel $(\Z/2\Z)^2$. We obtain maps $\pi_1: C\to E_1$ and $\pi_2: C\to E_2$. Now let $E$ be an optimal quotient of $C$ such that $\pi_1$ factors through $\pi: C\to E$. Then $E'=\mathrm{Ker}(J\to E)$ gives rise to a map $\pi': C\to E'$, which is an optimal quotient such that $J(C)\to E_2$ factors through $J(C)\to E'$. We obtain an isogeny $J(C)\to E\times E'$ with kernel contained in the kernel of $J(C)\to E_1\times E_2$. By \cite[Section 2]{kuhn}, the kernel of $J(C)\to E\times E'$ is $(\Z/d\Z)^2$, where $d$ is the degree of $C\to E$. We conclude that $d$ must equal 2. The degree 2 quotient $C\to E$ thus gives rise to an automorphism on $C$ of order 2 which is not the hyperelliptic involution. 
 \end{proof}
\noindent {\bf Proof of Proposition \ref{exclsetsprop}}. We consider $(i,j)=(1,2)$ and obtain by the previous proposition that $M\cdot P=P$, where $P=(A,\lambda,\al_2)$. By Theorem \ref{thmjacorprod},
%
%
%
 $(A,\lambda)$ is isomorphic over a degree 2 extension $L/K$ to the jacobian of a hyperelliptic curve $C/K$, where $L/K$ is an extension of degree at most 2.  Note that $M$ acts on $P=(A,\lambda,\al_2)$ by
 \[
 M\cdot (A,\lambda,\al_2)=(A,\lambda,M\cdot \al_2).
 \]
 Since $(A,\lambda,\al_2)=(A,\lambda,M\cdot \al_2)$ in the moduli space, then there must be an automorphism $\phi: (A,\lambda)\to(A,\lambda)$ satisfying $\phi^*\al_2=M\al_2$.
 %
 %
  By Torelli's theorem, $\phi$ arises from an automorphism $\psi: C\to C$, and $\psi$ commutes with the hyperelliptic involution $\iota: C\to C$. Let $p_1,\ldots,p_6$ be the Weierstrass points of $C$. Then
 \begin{eqnarray}
 \label{hypjac}
 J(C)[2] & = & \mathrm{Span}_{\F_2}\{[p_i-p_1] :i \in \{2,\ldots,6\}\} \\ 
 & = & \frac{\{\sum_i a_i\cdot p_i \in \bigoplus_{i=1}^6\F_2 p_i \mid \sum_i a_i =0\}}{\F_2(p_1+\ldots+p_6)}.
 \end{eqnarray}
 
Now $\psi$ must act on the set of Weierstrass points, and $\phi^2$ acts trivially on the 2-torsion. 

Let $\overline{G}=\mathrm{Aut}_{\overline{\Q}}(C)/\langle \iota\rangle$, where $\iota$ is the hyperelliptic involution. The elements of $\overline{G}$ are in 1-1 correspondence with automorphisms of $\P^1$ permuting the $x$-coordinates of $\mathcal{P}=\{p_1,\ldots,p_6\}$. Any permutation $\sigma\in S_{\mathcal{P}}$ acts on $J(C)[2]$ via (\ref{hypjac}) and preserves the Weil pairing. This gives rise to the exceptional isomorphism $S_6\simeq \mathrm{Sp}_4(\F_2)$. In particular, $\phi^2$ acting trivially on the 2-torsion implies that $\psi^2$ fixes the Weierstrass points, so $\psi^2\in \{\mathrm{Id},\iota\}$. Moreover, the conjugacy class of $M$ corresponds to a unique cycle type in $S_6$. Since $M$ is an involution and the conjugacy class of $M$ has size 15, $M$ has to correspond to either a product of 1 or 3 transpositions. If $\psi$ fixes four Weierstrass points, $M$ acts trivially on a 3-dimensional $\F_2$-subspace. However, $\mathrm{ker}(M-I_4)$ is 2-dimensional. We conclude that $M$ corresponds to a product of three transpositions.

 On $\P^1$, we may assume $\psi$ acts as $x\mapsto -x$ as $\psi$ has order 2 modulo $\langle \iota \rangle$. Since $\psi$ acts on the Weierstrass points as a product of three transpositions, $C$ must have a model of the form
\[
C:\; Y^2 = (X^2-a^2)(X^2-b^2)(X^2-c^2),\;\; a,b,c\in K.
\]
Here $\psi$ is given by $(X,Y)\mapsto (-X,\pm Y)$. In particular, $\psi$ is an involution, and we can apply the previous Lemma.

Conversely, starting with an isogeny $J(C)\to E_1\times E_2$ we can apply the previous lemma to see that $C$ must admit an order 2 automorphism unequal to the hyperelliptic involution. By \cite[Lemma 2]{shaska}, this automorphism acts as a product of three involutions on the Weierstrass points, so that it must act on $J(C)[2]$ as a conjugate $N$ of $M$, which shows that $N \cdot P = P$, as desired. 
%

\begin{remark}
\label{remarkexclusionsets}
One may wonder whether, or to what extent, the exceptional set in the application of Baker's method violates the obtained height bound. Since the points in this 2-dimensional exceptional set have three pairs of equal coordinates (up to sign), from their point of view there are not 10 but 7 ample divisors $D_i$. So a similar application of Baker's method would give us the same bound for the exceptional set if $|S|<7$, up to a further 1-dimensional exceptional set where another pair of coordinates is equal up to sign. Subsequently, $|S|<4$ allows us to bound this smaller set, up to a 0-dimensional exceptional set. Finally, for $|S|<3$ we have, of course, a stronger bound without exceptions using Runge's method.
\end{remark}

\section{List of G\"opel and azygous quadruples}
G\"opel quadruples:
\begin{alignat*}{2}
    &\{
        ( 0 0 1 1 ),
        ( 0 0 1 0 ),
        ( 1 0 0 1 ),
        ( 1 0 0 0 )
    \}, \quad &&
    \{
        ( 1 1 0 0 ),
        ( 0 0 1 1 ),
        ( 0 1 1 0 ),
        ( 1 0 0 1 )
    \}, \\ &
    \{
        ( 0 0 1 1 ),
        ( 0 0 0 1 ),
        ( 0 1 1 0 ),
        ( 0 1 0 0 )
    \}, &&
    \{
        ( 0 0 1 0 ),
        ( 1 1 1 1 ),
        ( 1 0 0 1 ),
        ( 0 1 0 0 )
    \}, \\ &
    \{
        ( 0 0 0 0 ),
        ( 0 1 1 0 ),
        ( 1 1 1 1 ),
        ( 1 0 0 1 )
    \}, &&\{
        ( 1 1 0 0 ),
        ( 0 0 0 1 ),
        ( 1 0 0 1 ),
        ( 0 1 0 0 )
    \}, \\ &
    \{
        ( 1 1 0 0 ),
        ( 0 0 1 0 ),
        ( 0 1 1 0 ),
        ( 1 0 0 0 )
    \}, &&
    \{
        ( 1 1 0 0 ),
        ( 0 0 0 1 ),
        ( 0 0 1 0 ),
        ( 1 1 1 1 )
    \}, \\ &
    \{
        ( 1 1 0 0 ),
        ( 0 0 1 1 ),
        ( 0 0 0 0 ),
        ( 1 1 1 1 )
    \}, &&
    \{
        ( 0 0 0 1 ),
        ( 0 0 0 0 ),
        ( 1 0 0 1 ),
        ( 1 0 0 0 )
    \}, \\
    &\{
        ( 0 0 0 1 ),
        ( 0 1 1 0 ),
        ( 1 1 1 1 ),
        ( 1 0 0 0 )
    \}, &&
    \{
        ( 0 0 1 1 ),
        ( 0 0 0 1 ),
        ( 0 0 1 0 ),
        ( 0 0 0 0 )
    \}, \\ &
    \{
        ( 0 0 1 0 ),
        ( 0 0 0 0 ),
        ( 0 1 1 0 ),
        ( 0 1 0 0 )
    \}, &&
    \{
        ( 1 1 0 0 ),
        ( 0 0 0 0 ),
        ( 1 0 0 0 ),
        ( 0 1 0 0 )
    \}, \\ &
    \{
        ( 0 0 1 1 ),
        ( 1 1 1 1 ),
        ( 1 0 0 0 ),
        ( 0 1 0 0 )
    \}.
\end{alignat*}
Azygous quadruples:
\begin{alignat*}{2} & \{
        ( 0 0 1 1 ),
        ( 0 0 1 0 ),
        ( 0 1 1 0 ),
        ( 1 1 1 1 )
    \}, \quad && \{
        ( 0 0 1 0 ),
        ( 0 0 0 0 ),
        ( 1 1 1 1 ),
        ( 1 0 0 0 )
    \}, \\ & \{
        ( 1 1 0 0 ),
        ( 0 0 0 1 ),
        ( 0 0 0 0 ),
        ( 0 1 1 0 )
    \}, && \{
        ( 1 1 0 0 ),
        ( 0 0 1 1 ),
        ( 0 0 0 1 ),
        ( 1 0 0 0 )
    \}, \\ & \{
        ( 0 1 1 0 ),
        ( 1 0 0 1 ),
        ( 1 0 0 0 ),
        ( 0 1 0 0 )
    \}, && \{
        ( 0 0 1 1 ),
        ( 0 0 0 0 ),
        ( 0 1 1 0 ),
        ( 1 0 0 0 )
    \}, \\
    & \{
        ( 1 1 0 0 ),
        ( 0 1 1 0 ),
        ( 1 1 1 1 ),
        ( 0 1 0 0 )
    \}, && \{
        ( 0 0 0 1 ),
        ( 0 0 1 0 ),
        ( 0 1 1 0 ),
        ( 1 0 0 1 )
    \}, \\
    & \{
        ( 0 0 0 1 ),
        ( 0 0 1 0 ),
        ( 1 0 0 0 ),
        ( 0 1 0 0 )
    \}, && \{
        ( 1 1 0 0 ),
        ( 0 0 1 1 ),
        ( 0 0 1 0 ),
        ( 0 1 0 0 )
    \}, \\
    & \{
        ( 0 0 1 1 ),
        ( 0 0 0 0 ),
        ( 1 0 0 1 ),
        ( 0 1 0 0 )
    \}, && \{
        ( 0 0 0 1 ),
        ( 0 0 0 0 ),
        ( 1 1 1 1 ),
        ( 0 1 0 0 )
    \}, \\
    & \{
        ( 1 1 0 0 ),
        ( 1 1 1 1 ),
        ( 1 0 0 1 ),
        ( 1 0 0 0 )
    \}, && \{
        ( 0 0 1 1 ),
        ( 0 0 0 1 ),
        ( 1 1 1 1 ),
        ( 1 0 0 1 )
    \}, \\
    & \{
        ( 1 1 0 0 ),
        ( 0 0 1 0 ),
        ( 0 0 0 0 ),
        ( 1 0 0 1 )
    \}. 
\end{alignat*}
\bibliographystyle{amsalpha}
\bibliography{bibliotdn}

\providecommand{\bysame}{\leavevmode\hbox to3em{\hrulefill}\thinspace}
\providecommand{\MR}{\relax\ifhmode\unskip\space\fi MR }
\providecommand{\MRhref}[2]{%
  \href{http://www.ams.org/mathscinet-getitem?mr=#1}{#2}
}
\providecommand{\href}[2]{#2}
\begin{thebibliography}{BvdGHZ08}

\bibitem[AVA17]{AbramovichVarillyAlvarado17}
Dan Abramovich and Anthony V\'{a}rilly-Alvarado, \emph{Level structures on
  abelian varieties and {V}ojta's conjecture}, Compos. Math. \textbf{153}
  (2017), no.~2, 373--394, With an appendix by Keerthi Madapusi Pera.
  \MR{3705229}

\bibitem[BG96]{BugeaudGyory96}
Yann Bugeaud and K\'{a}lm\'{a}n Gy\H{o}ry, \emph{Bounds for the solutions of
  unit equations}, Acta Arith. \textbf{74} (1996), no.~1, 67--80. \MR{1367579}

\bibitem[Bil95]{Bilu95}
Yuri Bilu, \emph{Effective analysis of integral points on algebraic curves},
  Israel Journal of Mathematics \textbf{90} (1995), no.~1, 235--252.

\bibitem[{Bug}18]{BugeaudLinearForms}
Yann {Bugeaud}, \emph{{Linear forms in logarithms and applications.}}, vol.~28,
  Z\"urich: European Mathematical Society (EMS), 2018 (English).

\bibitem[BvdGHZ08]{BruiniervdGZagier}
Bruinier, van~der Geer, Harder, and Zagier, \emph{The 1-2-3 of modular forms},
  Universitext, Springer-Verlag, Berlin, 2008.

\bibitem[CM20]{ClingherMalmandier20}
Adrian {Clingher} and Andreas {Malmandier}, \emph{Normal {F}orms for {K}ummer
  {S}urfaces}, London Mathematical Society Lecture Note Series, vol.~2,
  p.~119–174, Cambridge University Press, 2020.

\bibitem[Deb99]{Debarre99}
Olivier Debarre, \emph{Tores et variétés abéliennes complexes}, EDP
  Sciences, 1999.

\bibitem[FC90]{ChaiFaltings}
Gerd Faltings and Ching-Li Chai, \emph{Degeneration of abelian varieties},
  Springer-Verlag, 1990.

\bibitem[Igu64]{Igusa64}
Jun-ichi Igusa, \emph{On the graded ring of theta-constants}, Amer. J. Math.
  \textbf{86} (1964), 219--246. \MR{0164967}

\bibitem[Igu72]{IgusaThetaF}
Jun-Ichi Igusa, \emph{Theta {F}unctions}, Springer-Verlag, 1972.

\bibitem[Kuh88]{kuhn}
Robert~M. Kuhn, \emph{Curves of genus {$2$} with split {J}acobian}, Trans.
  Amer. Math. Soc. \textbf{307} (1988), no.~1, 41--49. \MR{936803}

\bibitem[Lev08]{Levin08}
Aaron Levin, \emph{Variations on a theme of {R}unge: effective determination of
  integral points on certain varieties}, J. Théor. Nombres Bordeaux (2008),
  385--417.

\bibitem[Lev14]{Levin14}
\bysame, \emph{Linear forms in logarithms and integral points on
  higher-dimensional varieties}, Algebra and Number Theory \textbf{8} (2014),
  647--687.

\bibitem[Lev18]{Levin18}
\bysame, \emph{Extending {R}unge's method for integral points}, Higher genus
  curves in mathematical physics and arithmetic geometry, Contemp. Math., vol.
  703, Amer. Math. Soc., Providence, RI, 2018, pp.~171--188. \MR{3782466}

\bibitem[LF17]{LeFourn4}
Samuel Le~Fourn, \emph{Sur la m\'ethode de {R}unge et les points entiers de
  certaines vari\'et\'es modulaires de {S}iegel}, C. R. Math. Acad. Sci. Paris
  \textbf{355} (2017), no.~8, 847--852.

\bibitem[LF19]{LeFourn3}
\bysame, \emph{{A tubular variant of Runge's method in all dimensions, with
  applications to integral points on Siegel modular varieties.}}, {Algebra
  Number Theory} \textbf{13} (2019), no.~1, 159--209 (English).

\bibitem[LF20]{LeFourn5}
\bysame, \emph{{Tubular approaches to Baker's method for curves and
  varieties.}}, {Algebra Number Theory} \textbf{14} (2020), no.~3, 785--807
  (English).

\bibitem[{Mum}07]{MumfordTataII}
David {Mumford}, \emph{{Tata lectures on theta. II: Jacobian theta functions
  and differential equations. With the collaboration of C. Musili, M. Nori, E.
  Previato, M. Stillman, and H. Umemura. Reprint of the 1984 edition.}},
  reprint of the 1984 edition ed., Basel: Birkh\"auser, 2007 (English).

\bibitem[NU73]{NamikawaUeno73}
Yukihiko {Namikawa} and Kenji {Ueno}, \emph{{The complete classification of
  fibres in pencils of curves of genus two.}}, {Manuscr. Math.} \textbf{9}
  (1973), 143--186 (English).

\bibitem[OU73]{OortUeno}
F.~Oort and K.~Ueno, \emph{Principally polarized abelian varieties of dimension
  two or three are {J}acobian varieties}, J. Fac. Sci. Univ. Tokyo Sect. IA
  Math. \textbf{20} (1973), 377--381.

\bibitem[Paz12]{Pazuki12b}
Fabien Pazuki, \emph{Theta height and {F}altings height}, Bull. Soc. Math. Fr.
  \textbf{1} (2012), 19--49.

\bibitem[Sil09]{SilvermanAEC}
Joseph Silverman, \emph{The {A}rithmetic of {E}lliptic {C}urves}, vol. 106,
  Springer-Verlag, 2009.

\bibitem[Str10]{Strengthesis}
Marco Streng, \emph{Complex multiplication of abelian surfaces}, PhD Thesis,
  University of Leiden, https://openaccess.leidenuniv.nl/handle/1887/15572,
  2010.

\bibitem[SV04]{shaska}
Tanush Shaska and Helmut V\"{o}lklein, \emph{Elliptic subfields and
  automorphisms of genus 2 function fields}, Algebra, arithmetic and geometry
  with applications ({W}est {L}afayette, {IN}, 2000), Springer, Berlin, 2004,
  pp.~703--723. \MR{2037120}

\bibitem[vdG82]{vdG82}
Gerard van~der Geer, \emph{On the geometry of a {S}iegel modular threefold},
  Math. Ann. \textbf{260} (1982), no.~3, 317--350.

\bibitem[vK14]{vonKanel14}
Rafael von K\"anel, \emph{{An effective proof of the hyperelliptic Shafarevich
  conjecture.}}, {J. Th\'eor. Nombres Bordx.} \textbf{26} (2014), no.~2,
  507--530 (English).

\end{thebibliography}

\end{document}